\documentclass[11pt,reqno]{amsart}
\usepackage[margin=1in]{geometry}
\usepackage{amsfonts,amsmath,amssymb,wasysym,stmaryrd,enumerate}    
\usepackage{amsthm}
\usepackage{mathrsfs}  
\usepackage{setspace}
\usepackage{bm}
\onehalfspace
\usepackage{tikz}    
\usepackage{comment}
\usetikzlibrary{arrows,matrix,decorations.markings,shapes.geometric}    

\numberwithin{equation}{section}    

\theoremstyle{plain}
\newtheorem{thm}{Theorem}[section]
\newtheorem{lem}[thm]{Lemma}
\newtheorem{prop}[thm]{Proposition}
\newtheorem{cor}[thm]{Corollary}
\newtheorem{conj}[thm]{Conjecture}
\theoremstyle{definition}

\newtheorem{exmp}[thm]{Example}
\theoremstyle{remark}
\newtheorem{rem}[thm]{Remark}

\newtheorem*{rem*}{Remark}

\newcommand{\bs}{\boldsymbol}
\newcommand{\be}{\begin{equation}}    
\newcommand{\ee}{\end{equation}}    
\newcommand{\beu}{\begin{equation*}}    
\newcommand{\eeu}{\end{equation*}}    
\newcommand{\bea}{\begin{eqnarray}}    
\newcommand{\eea}{\end{eqnarray}}    
\newcommand{\beaa}{\begin{eqnarray*}}    
\newcommand{\eeaa}{\end{eqnarray*}}    
\newcommand{\bmx}{\begin{pmatrix}}    
\newcommand{\emx}{\end{pmatrix}}

\newcommand{\mf}{\mathfrak}
\newcommand{\mc}{\mathcal}
\newcommand{\wt}{\widetilde}    
\newcommand{\al}{{\alpha}}    

\newcommand{\nn}{\nonumber}

\newcommand{\la}{\lambda}    
\newcommand{\Z}{{\mathbb Z}}

\newcommand{\C}{{\mathbb C}}

\newcommand{\Wr}{{\rm Wr}} 

\newcommand{\id}{{\mathrm{id}}}

\DeclareMathOperator{\ord}{ord}

\newcommand{\gl}{\mf{gl}}

\DeclareMathOperator{\End}{End}

\newcommand{\HH}{\mathcal H}

\author{Chenliang Huang} 
\address{Department of Mathematical Sciences, 402
  N. Blackford St, LD 270, IUPUI, Indianapolis, IN 46202, USA.  }
\email{ch30@iupui.edu}

\author{Evgeny Mukhin} 
\address{Department of Mathematical Sciences, 402
  N. Blackford St, LD 270, IUPUI, Indianapolis, IN 46202, USA.  }
\email{mukhin@math.iupui.edu}

\author{Beno\^{\i}t Vicedo}
\address{Department of Mathematics, University of York, York YO10 5DD, UK.}
\email{benoit.vicedo@gmail.com}  

\author{Charles Young}
\address{School of Physics, Astronomy and Mathematics, University of Hertfordshire, College Lane, Hatfield AL10 9AB, UK.}
\email{c.a.s.young@gmail.com}

\title[ Bethe ansatz equation and rational pseudodifferential operators]{The solutions of $\mathfrak{gl}_{M|N}$ Bethe ansatz equation\\[1mm]
and rational pseudodifferential operators}
\begin{document}
\maketitle
\begin{abstract}
We describe a reproduction procedure which, given a solution of the $\mathfrak{gl}_{M|N}$ Gaudin Bethe ansatz equation associated to a tensor product of polynomial modules, produces a family $P$ of other solutions called the population. 
To a population we associate a rational pseudodifferential operator $R$ and a superspace $W$ of rational functions. 

We show that if at least one module is typical then the population $P$ is canonically identified with the set of minimal factorizations of $R$ and with the space of full superflags in $W$.

We conjecture that the singular eigenvectors (up to rescaling) of all $\mathfrak{gl}_{M|N}$ Gaudin Hamiltonians are in a bijective correspondence with certain superspaces of rational functions.

\end{abstract}
	
\section{Introduction}
We study the Gaudin model associated to tensor products of polynomial modules over the Lie superalgebra $\gl_{M|N}$. The main method is the Bethe ansatz; see \cite{MVY14}. It is well-known that  the Bethe ansatz method in its straightforward formulation is incomplete -- it does not provide the full set of eigenvectors of the Hamiltonians; see \cite{MVcounterexample}. In this paper we propose a regularization of the Bethe ansatz method, drawing our inspiration from \cite{MVpopulations}.
	
In the case of Lie algebras, the regularization of the Bethe ansatz is obtained by the identification of the spectrum of the model with opers -- linear differential operators with appropriate properties \cite{RybnikovProof, FFR}. In the case of $\gl_M$, the opers are reduced to scalar linear differential operators of order $M$ with polynomial kernels. The spaces of polynomials of dimension $M$ obtained this way are intersection points of Schubert varieties whose data is described by the parameters of the Gaudin model. Moreover, the action of the algebra of Gaudin Hamiltonians can be identified with the regular representation of the scheme-theoretic intersection algebra, \cite{MTVschubert}.

\medskip

We argue that in the case of the Lie superalgebra $\gl_{M|N}$ one should study rational pseudodifferential operators and appropriate spaces of rational functions which we call $\gl_{M|N}$ spaces. 

\medskip

Let us describe our findings in more detail. The $\gl_{M|N}$ Gaudin model depends on the choice of a sequence of polynomial representations, each equipped with distinct complex evaluation parameters. The Bethe ansatz depends on a choice of Borel subalgebra. Such a choice is equivalent to the choice of a parity sequence $\bs s=(s_1,\dots,s_{M+N})$, $s_i\in\{\pm1\}$. The highest weights of representations and the evaluation parameters are encoded into polynomials $T_i^{\bs s}$ (see \eqref{T fun}).  A solution of the Bethe ansatz equation is represented by a sequence of monic polynomials $(y_1,\dots, y_{M+N-1})$, so that the roots of $y_i$ are Bethe variables corresponding to the $i$th simple root (see \eqref{y fun}). 

The key ingredient is the reproduction procedure (see Theorem \ref{GRP}), which given a solution of the Bethe ansatz equation (BAE) produces a family of new solutions along a simple root. If the simple root is even, then the BAE means that the kernel of the operator
$$
\bigg(\partial -\ln'\frac{T_i^{\bs s}y_{i-1}y_{i+1}}{T_{i+1}^{\bs s}y_i} \bigg)\bigg(\partial-\ln' y_i\bigg)
$$
consists of polynomials. Then one shows that all tuples of the form $(y_1,\dots, \wt{y}_i,\dots, y_{M+N-1})$, where  $\wt{y}_i$ is any (generic) polynomial in the kernel of the differential operator, represent solutions of the BAE. This gives the bosonic reproduction procedure, which was described in \cite{MVpopulations}. 

If the simple root is odd then the BAE means that $y_i$ divides a certain explicit polynomial $\mc N$ and it turns out that the tuple $(y_1,\dots, \wt{y}_i,\dots, y_{M+N-1})$, $\wt{y}_i=\mc N/y_i$, again satisfies the BAE (if generic). This gives the fermionic reproduction procedure. Moreover, the fermionic reproduction can be rewritten as an equality of rational pseudodiffential operators (assuming $s_i=1$):
$$
\bigg(\partial-\ln'\frac{T_i^{\bs s}y_{i-1}}{y_i}\bigg)\bigg(\partial-\ln'\frac{y_{i+1}}{T_{i+1}^{\bs s}y_i}\bigg)^{-1}=
\bigg(\partial-\ln'\frac{\wt{y}_{i}}{T_i^{\wt {\bs s }} y_{i-1}}\bigg)^{-1}\bigg(\partial-\ln'\frac{T_{i+1}^{\wt{\bs s}} \wt{y}_i}{y_{i+1}}\bigg),
$$
where $\wt{\bs s}=(s_1,\dots, s_{i+1},s_i,\dots, s_{M+N})$.

The bosonic and fermionic procedures are very different in nature. The bosonic procedure describes a one-parameter family of solutions of the BAE. However, these solutions are not physical: $\deg \wt {y}_i$ is large and the corresponding Bethe vector is zero on weight grounds. The fermionic procedure produces only one new solution. Moreover, in contrast to the bosonic case, the new BAE corresponds to a new choice of the Borel subalgebra. If the original solution produced an eigenvector which was singular with respect to the original Borel subalgebra, the new solution produces the eigenvector in the same isotypical component but singular with respect to a new Borel subalgebra. The two eigenvectors are related by the diagonal action of $\gl_{M|N}$.\footnote{These features are reminiscent of trigonometric Gaudin models and Gaudin with quasi-periodic boundary conditions \cite{MVquasipolynomials}, in which the diagonal symmetry is broken. In those cases reproduction produces one new solution, which describes the same eigenvector (up to proportionality) but with respect to a different Borel subalgebra.}

The most important feature of the bosonic and fermionic procedures is the conservation of the eigenvalues of the Gaudin Hamiltonians written in terms of the Bethe roots (see Lemma \ref{gl11wf}). We call the set of all solutions obtained by repeated applications of the reproduction procedures a \emph{population}.

\medskip

We define a rational pseudodifferential operator $R$ (see \eqref{RPDO of Pop}). In the standard parity $\bs s_0=(1,\dots,1,-1,\dots,-1)$, it has the form:
$R=D_{\bar 0} (D_{\bar 1})^{-1}$, where $D_{\bar 0}$,  $D_{\bar 1}$ are scalar differential operators of orders $M$ and $N$ with rational coefficients, given by:
\begin{align*}
&D_{\bar 0}=\bigg(\partial-\ln'\frac{T_1^{\bs s_0}y_0}{y_1}\bigg)\bigg(\partial-\ln'\frac{T_2^{\bs s_0}y_1}{y_2}\bigg)
\dots\bigg(\partial-\ln'\frac{T_M^{\bs s_0}y_{M-1}}{y_M}\bigg), \\
&D_{\bar 1}=\bigg( \partial-\ln'\frac{y_{M+N}}{T_{M+N-1}^{\bs s_0}y_{M+N-1}}\bigg)\dots\bigg(\partial-\ln'\frac{y_{M+2}}{T_{M+2}^{\bs s_0}y_{M+1}}\bigg)
\bigg(\partial-\ln'\frac{y_{M+1}}{T_{M+1}^{\bs s_0}y_{M}}\bigg).
\end{align*}
(Here we set $y_0=y_{M+N}=1$.)
We show that $R$ does not change under reproduction procedures (see Theorem \ref{Diff of Popu}) and, moreover, if at least one weight is typical, then 
the population is identified with the set of all minimal factorizations of $R$ into linear factors (see Theorem \ref{operator to population}).

\medskip

Then we study the space $W=V\oplus U$, where $V=\ker D_{\bar 0}$, $U=\ker D_{\bar 1}$. We show that if at least one weight is typical, then $U\cap V=0$. We think of $W$ as a superspace of dimension $M+N$, with even part $V$ and odd part $U$. We identify the population with the space of all full superflags in $W$ (see Theorem \ref{operator to population}).

The operators $D_{\bar 0}$ and  $D_{\bar 1}$ up to a conjugation coincide with $\gl_M$ and $\gl_N$ operators. 
It follows that $W$ consists of rational functions. In other words, $W$ is given by a pair of spaces of polynomials with prescribed ramification conditions linked via polynomials $y_M, T_M,T_{M+1}$. This leads us to a definition of a \emph{$\gl_{M|N}$ space} (see Section \ref{glM|N space}). The Gaudin Hamiltonians acting in tensor products of polynomial modules belong to a natural commutative algebra $\mc B(\bs \la)$ of higher Gaudin Hamiltonians. We conjecture that the joint eigenvectors of this algebra $\mc B(\bs \la)$ are parametrized by $\gl_{M|N}$ spaces  (see Conjecture \ref{the conj}).

	\medskip
	
	The paper is constructed as follows. In Sections \ref{sec: glMN}--\ref{sec: ba} we recall various facts and set up our notation:
	Section \ref{sec: glMN} is devoted to the Lie superalgebra $\gl_{M|N}$, Section \ref{sec: rpdos} to rational pseudodifferential operators, and 
	Section \ref{sec: ba} to the Gaudin model and the Bethe ansatz. In Section \ref{sec: rep} we recall the $\gl_2$ (bosonic) reproduction procedure and define the $\gl_{1|1}$ (fermionic) reproduction procedure. In Section \ref{sec: repglmn} we define the $\gl_{M|N}$ reproduction procedure and the rational pseudodifferential operator of a population. In Section \ref{sec: spaces} we define $\gl_{M|N}$ spaces of rational functions, and give the identification of superflags, minimal factorizations and points of a population (see Theorem \ref{operator to population}). In Section \ref{sec: conj} we give various conjectures and examples.
	
	\medskip
	
	{\bf Acknowledgments.} The research of EM is
partially supported by a grant from the Simons Foundation \#353831. CY is grateful to the Department of Mathematical Sciences, IUPUI, for hospitality during his visit in September 2017 when part of this work was completed. 
	
	\section{Preliminaries on $\gl_{M|N}$}\label{sec: glMN}
	Fix $M,N\in\Z_{\geq 0}$. In this section, we will recall some facts about $\gl_{M|N}$. For details see, for example, \cite{ChengWang}.

	\subsection{Lie superalgebra $\gl_{M|N}$ }
	A \emph{vector superspace} $V=V_{\bar{0}}\oplus V_{\bar{1}}$ is a $\Z_2$-graded vector space. The \emph{parity} of a homogeneous vector $v$ is denoted by $|v|\in\Z/2\Z=\{\bar{0},\bar{1}\}$. We set $(-1)^{\bar{0}}=1$ and $(-1)^{\bar{1}}=-1$. An element $v$ in $V_{\bar{0}}$ (respectively $V_{\bar{1}}$) is called \emph{even} (respectively \emph{odd}), and we write $|v|=\bar{0}$ (respectively $|v|=\bar{1}$). Let $\C^{M|N}$ be a complex vector superspace, with $\dim(\C^{M|N})_{\bar{0}}=M$ and $\dim(\C^{M|N})_{\bar{1}}=N$. Choose a homogeneous basis $e_i$, $i=1,\dots,M+N$, of $\C^{M|N}$ such that $|e_i|=\bar{0}$, $i=1,\dots, M$, and $|e_i|=\bar{1}$, $i=M+1,\dots, M+N$. Set $|i|=|e_i|$.

	Let $\bs s=(s_1,\dots,s_{M+N})$, $s_i\in\{\pm1\}$, be a sequence such that $1$ occurs exactly $M$ times. We call such a sequence a \emph{parity sequence}. We call the parity sequence $\bs s_0=(1,\dots,1,-1,\dots,-1)$ \emph{standard}. Denote the set of all parity sequences by $S_{M|N}$. The order of $S_{M|N}$ is $\binom{M+N}{M}$.
	The set $S_{M|N}$ is identified with $\mathfrak{S}_{M+N}/(\mathfrak{S}_M\times\mathfrak{S}_N)$, where $\mathfrak{S}_k$ denotes the permutation group of $k$ letters. We fix a lifting $S_{M|N}=\mathfrak{S}_{M+N}/(\mathfrak{S}_M\times\mathfrak{S}_N)\to \mathfrak{S}_{M+N}$: for each $\bs s\in S_{M|N}$, we define $\sigma_{\bs s}\in\mathfrak{S}_{M+N}$ by
	\[\sigma_{\bs s}(i)=\begin{cases}
	\#\{j\;|\; j\leq i, \ s_j=1\} & \mbox{if } s_i=1,\\
	M+\#\{j\;|\;j\leq i,\ s_j=-1\} & \mbox{if } s_i=-1.
	\end{cases}\] Note that $\sigma_{\bs s_0}=\id$ and $(-1)^{|\sigma_{\bs s}(i)|}=s_i$. (The element $\sigma_{\bs s}$ is sometimes called an \emph{unshuffle}.)
	
	For a parity sequence $\bs s\in S_{M|N}$ and $i=1,\dots, M+N$, define numbers 
	$$
	\bs s_i^+=\#\{j\;|\; j>i, \ s_j=1\}, \qquad \bs s_i^-=\#\{j\;|\; j<i,\ s_j=-1\}.
	$$
	We have 
	$$
	\bs s_i^+=\begin{cases} M-\sigma_{\bs s}(i) & \mbox{if } s_i=1,\\  \sigma_{\bs s}(i)-i  & \mbox{if } s_i=-1, \end{cases}
	\qquad 
	\bs s_i^-=\begin{cases}i-\sigma_{\bs s}(i) & \mbox{if } s_i=1,\\  \sigma_{\bs s}(i)-M-1  & \mbox{if } s_i=-1. \end{cases}
	$$
	
	\medskip
	
	The \emph{Lie superalgebra $\gl_{M|N}$} is spanned by $e_{ij}$, $i,j=1,\dots, M+N$, with $|e_{ij}|=|i|+|j|$, and the superbracket is given by 
	\[[e_{ij},e_{kl}]=\delta_{jk}e_{il}-(-1)^{(|i|+|j|)(|k|+|l|)}\delta_{il}e_{kj}.\]
	The universal enveloping algebra of $\gl_{M|N}$ is denoted by $U\gl_{M|N}$. 
	
	There is a non-degenerate invariant bilinear form $(\ ,\ )$ on $\gl_{M|N}$, such that \[
	(e_{ab},e_{cd})=(-1)^{|a|}\delta_{ad}\delta_{bc}.
	\] 
	The \emph{Cartan subalgebra} $\mathfrak{h}$ of $\gl_{M|N}$ is spanned by $e_{ii}$, $i=1,\dots, M+N$. The \emph{weight space} $\mathfrak{h}^*$ is the dual space of $\mathfrak{h}$. Let $\epsilon_i$, $i=1,\dots, M+N$, be a basis of $\mathfrak{h}^*$, such that $\epsilon_i(e_{jj})=\delta_{ij}$. The bilinear form $(\ ,\ )$ is extended to $\mathfrak{h}^*$ such that $(\epsilon_i,\epsilon_j)=(-1)^{|i|}\delta_{ij}$. The \emph{root system} $\Phi$ is a subset of $\mathfrak{h}^*$ given by \[
	\Phi=\{\epsilon_i-\epsilon_j\;|\;i,j=1,\dots, M+N\mbox{ and } i\neq j\}.
	\] 
	A root $\epsilon_i-\epsilon_j$ is called \emph{even} (respectively \emph{odd}), if $|i|=|j|$ (respectively $|i|\neq|j|$). 
	\subsection{Root systems}
	For each parity sequence $\bs s\in S_{M|N}$, define the set of \emph{$\bs s$-positive roots} $\Phi^+_{\bs s}=\{\epsilon_{\sigma_{\bs s}(i)}-\epsilon_{\sigma_{\bs s}(j)}\;|\;i,j=1,\dots,M+N\mbox{ and }i< j\}$. Define the \emph{$\bs s$-positive simple roots}  $\alpha^{\bs s}_{i}=\epsilon_{\sigma_{\bs s}(i)}-\epsilon_{\sigma_{\bs s}(i+1)}$, $i=1,\dots, M+N-1$. Define 
	\[
	e^{\bs s}_{ij}=e_{\sigma_{\bs s}(i),\sigma_{\bs s}(j)},\; i,j=1,\dots,M+N.
	\] 
	The \emph{nilpotent subalgebra} $\mathfrak{n}^+_{\bs s}$ of ${\gl_{M|N}}$ (respectively $\mathfrak{n}^-_{\bs s}$) associated to $\bs s$, is generated by $\{e^{\bs s}_{i,i+1}\;|\;i=1,\dots,M+N-1\}$ (respectively $\{e^{\bs s}_{i+1,i}\;|\; i=1,\dots,M+N-1\}$). The algebra $\mathfrak{n}_{\bs s}^+$ (respectively $\mathfrak{n}_{\bs s}^-$) has a basis $\{e^{\bs s}_{ij}\;|\;i<j\}$ (respectively $\{e^{\bs s}_{ij}\;|\;i>j\}$). The Borel subalgebra associated to $\bs s$, is $\mathfrak{b}_{\bs s}=\mathfrak{h}\oplus\mathfrak{n}^+_{\bs s}$. We call the Borel subalgebra $\mathfrak{b}_{\bs s_0}$ \emph{standard}.
	
	In what follows, many objects depend on a parity sequence $\bs s$. If $\bs s$ is omitted from the notation, then it means the standard parity sequence. For example, we abbreviate $\mathfrak{n}_{\bs s_0}^+$, $\mathfrak{n}_{\bs s_0}^-$, and $\mathfrak{b}_{\bs s_0}$ to $\mathfrak{n}^+$, $\mathfrak{n}^-$, and $\mathfrak{b}$, respectively.
	
	\begin{exmp}
		Consider the case of $\gl_{3|3}$.  Two possible parity sequences from $S_{3|3}$ are: \newline $\bs s_1=(1,1,-1,-1,-1,1)$ and $\bs s_2=(1,-1,1,-1,1,-1)$. We have \[\sigma_{\bs s_1}=\bmx 1 & 2 & 3 & 4 & 5 & 6 \\ 1 & 2 & 4 & 5 & 6 & 3\emx,\quad\sigma_{\bs s_2}=\bmx 1 & 2 & 3 & 4 & 5 & 6 \\ 1 & 4 & 2 & 5 & 3 & 6\emx.\] The $\bs s_1$-positive simple roots and $\bs s_2$-positive simple roots are given respectively by \begin{align*}
		&(\alpha^{\bs s_1}_1,\alpha^{\bs s_1}_2,\alpha^{\bs s_1}_3,\alpha^{\bs s_1}_4,\alpha^{\bs s_1}_5)=(\epsilon_1-\epsilon_2,\epsilon_2-\epsilon_4,\epsilon_4-\epsilon_5,\epsilon_5-\epsilon_6,\epsilon_6-\epsilon_3),\\
		&(\alpha^{\bs s_2}_1,\alpha^{\bs s_2}_2,\alpha^{\bs s_2}_3,\alpha^{\bs s_2}_4,\alpha^{\bs s_2}_5)=(\epsilon_1-\epsilon_4,\epsilon_4-\epsilon_2,\epsilon_2-\epsilon_5,\epsilon_5-\epsilon_3,\epsilon_3-\epsilon_6).
		\end{align*}
	\end{exmp}
	
	We have \[(\alpha^{\bs s}_{i},\alpha^{\bs s}_{j})=(s_i+s_{i+1})\delta_{i,j}-s_i\delta_{i,j+1}-s_{i+1}\delta_{i+1,j}.\] The \emph{symmetrized Cartan matrix} associated to $\bs s$, $\left((\alpha^{\bs s}_{i},\alpha^{\bs s}_{j})\right)_{i,j=1}^{M+N-1}$, is described by the blocks 
	\be 
	\begin{matrix} 
		& \bmx (\alpha^{\bs s}_{i},\alpha^{\bs s}_{i}) & (\alpha^{\bs s}_{i},\alpha^{\bs s}_{i+1}) \\ (\alpha^{\bs s}_{i+1},\alpha^{\bs s}_{i}) & (\alpha^{\bs s}_{i+1},\alpha^{\bs s}_{i+1}) \emx= & \bmx s_i+s_{i+1} & -s_{i+1} \\ -s_{i+1} & s_{i+1}+s_{i+2} \emx.
		\nn 
	\end{matrix}
	\ee
	Explicitly, this block is one of the following cases depending on $(s_i,s_{i+1},s_{i+2})$:
	\be 
	\begin{matrix}  (1,1,1) & (1,1,-1) & (1,-1,1) & (-1,1,1)\\ 
		\bmx 2 & -1 \\ -1 & 2 \emx, & \bmx 2 & -1 \\ -1 & 0 \emx, & \bmx 0 & 1 \\ 1 & 0 \emx, & \bmx 0 & -1 \\ -1 & 2 \emx,\\ \\
		(-1,-1,-1) & (-1,-1,1) & (-1,1,-1) & (1,-1,-1)\\
		\bmx -2 & 1 \\ 1 & -2 \emx, & \bmx -2 & 1 \\ 1 & 0 \emx, & \bmx 0 & -1 \\ -1 & 0 \emx, &\;\; \bmx 0 & 1 \\ 1 & -2 \emx.
		\nn 
	\end{matrix}
	\ee
	
	\subsection{Representations of $\gl_{M|N}$}\label{representations}
	Let $V$ be a $\gl_{M|N}$ module. Given a parity sequence $\bs s\in S_{M|N}$ and a weight $\lambda\in\mathfrak{h}^*$, a non-zero vector $v^{\bs s}_{\lambda}\in V$ is called an \emph{$\bs s$-singular vector} of weight $\lambda$ if $\mathfrak{n}^+_{\bs s}v^{\bs s}_{\lambda}=0$ and $hv^{\bs s}_{\lambda}=\lambda(h)v^{\bs s}_{\lambda}$, for all $h\in\mathfrak{h}$. Denote the subspace of $\bs s$-singular vectors by $V^{\textup{$\bs s$ing}}$. Denote by $V_{\lambda}$ the subspace of vectors of weight $\lambda$, $V_{\lambda}=\{v\in V\;|\;hv=\lambda(h)v\mbox{, for all } h\in\mathfrak{h}\}$. Denote by $V^{\textup{$\bs s$ing}}_{\lambda}$ the subspace of $\bs s$-singular vectors of weight $\lambda$. Denote the subspaces of $\bs s_0$-singular vectors and of $\bs s_0$-singular vectors of weight $\la$ by $V^{sing}$ and $V^{sing}_\la$ respectively. Let $L^{\bs s}(\lambda)$ be the \emph{$\bs s$-highest weight irreducible module} of highest weight $\lambda$, generated by the $\bs s$-singular vector $v^{\bs s}_{\lambda}$. The $\bs s$-singular vector $v^{\bs s}_{\lambda}\in L^{\bs s}(\lambda)$ is called the \emph{$\bs s$-highest weight vector}. Denote by \[\lambda_{[\bs s]}=(\lambda_{[\bs s],1},\dots,\lambda_{[\bs s],M+N})=\left(\lambda(e^{\bs s}_{11}),\dots,\lambda(e^{\bs s}_{M+N,M+N})\right)\] the coordinate sequence of $\lambda$ associated to $\bs s$. We also use the notation $L^{\bs s}(\lambda_{[\bs s]})$ for $L^{\bs s}(\lambda)$.
	
	\begin{exmp}
		The superspace $\C^{M|N}$ is a $\gl_{M|N}$ module with the action given by $e_{ij}e_k=\delta_{j,k}e_i$. We have $\C^{M|N}\cong L^{\bs s}(1,0,\dots,0)=L^{\bs s}(\epsilon_{\sigma_{\bs s}(1)})$ for any $\bs s\in S_{M|N}$. 
		The $\bs s$-highest weight vector is $v^{\bs s}_{\epsilon_{\sigma_{\bs s}(1)}}=e_{\sigma_{\bs s}(1)}$. We call $\C^{M|N}$ the \emph{vector representation}.
	\end{exmp} 
	
	A module $V$ is called a \emph{polynomial module} if it is an irreducible submodule of $(\C^{M|N})^{\otimes n}$ for some $n\in \Z_{\geq 0}$.
	A highest weight module $L(\lambda)$ with respect to the standard Borel subalgebra $\mathfrak{b}$, is a polynomial module if 
and only if 
the weight $\lambda$ satisfies $\lambda_i\in\Z_{\geq 0}$ for all $i$, $\lambda_1\geq\dots\geq\lambda_M$, $\lambda_{M+1}\geq\dots\geq\lambda_{M+N}$, and $\lambda_M\geq \#\{i\;|\; \lambda_{M+i}\neq 0\;|\; i=1,\dots,N\}$. A weight $\lambda$ is called a \emph{polynomial weight} if $L(\lambda)$ is a polynomial module. It is known that the category of polynomial modules is a semisimple tensor category. 
	
	Let $\mu=(\mu_1\geq \mu_2\geq \dots)$ be a partition: $\mu_i\in\Z_{\geq 0}$ and $\mu_i=0$ if $i\gg 0$. The partition $\mu$ is called an \emph{$(M|N)$-hook partition} if $\mu_{M+1}\leq N$. Polynomial modules are parametrized by $(M|N)$-hook partitions.
	
	Let $L(\lambda)$ be a polynomial module with highest weight vector $v_\lambda$. Let $\bs s$ be a parity sequence. Then $L(\lambda)$ is isomorphic to an irreducible $\bs s$-highest weight module $L^{\bm{s}}(\lambda^{\bs s})$. The coordinate sequence $\lambda^{\bs s}_{[\bs s]}$ and the $\bs s$-highest weight vector $v^{\bs s}_\lambda$ can be found recursively as follows.
	
	Let $\bs s^{[i]}=(s_1,\dots,s_{i+1},s_i,\dots,s_{M+N})$ be the parity sequence obtained from $\bs s$ by switching the $i$-th and $(i+1)$-st coordinates. If $s_i\neq s_{i+1}$, then we have
\begin{equation}\label{si}
	\la_{[\bs s^{[i]}]}^{\bs s^{[i]}}=(\la_{[\bs s],1}^{\bs s},\dots,
	\la_{[\bs s],i-1}^{\bs s},  \la_{[\bs s],i+1}^{\bs s}+\delta,\la_{[\bs s],i}^{\bs s}-\delta, \la_{[\bs s],i+2}^{\bs s}, \dots, \la_{[\bs s],M+N}^{\bs s}),\qquad v^{\bs s^{[i]}}_{\la^{\bs s^{[i]}}}=(e^{\bs s}_{i+1,i})^\delta v^{\bs s}_{\la^{\bs s}},
\end{equation}
	where $\delta=1$ if $\la_{[\bs s],i}^{\bs s}+\la_{[\bs s],i+1}^{\bs s}\neq 0$ and $\delta=0$ otherwise.

	The following example illustrates how the coordinate sequence $\lambda^{\bs s}_{[\bs s]}$ can be found from an $(M|N)$-hook partition, and how the $\bs s$-highest weight vector $v^{\bs s}_\lambda$ is related to the highest weight vector $v_\lambda$.
	
	\begin{exmp}\label{ex1} Let $\mu=(7,6,4,3,3)$ be a $(3|3)$-hook partition. Choose some parity sequences:
		\be \bs s_0=(1,1,1,-1,-1,-1),\quad \bs s_1=(1,1,-1,-1,-1,1),\quad \bs s_2=(1,-1,1,-1,1,-1). \nn\ee
		The highest weights and the highest weight vectors for those choices can be read as:
		\be\begin{matrix}
			&\qquad \begin{tikzpicture}[baseline=-25pt,shorten >=-4pt,shorten <= - 4pt,scale=.4,yscale=-1,every node/.style={minimum size=.4cm,inner sep=0mm,draw,gray,rectangle}]
			\foreach \x in {1,2,3,4,5,6,7} {\node at (\x,1) {};}
			\foreach \x in {1,2,3,4,5,6} {\node at (\x,2) {};}
			\foreach \x in {1,2,3,4} {\node at (\x,3) {};}
			\foreach \x in {1,2,3} {\node at (\x,4) {};}
			\foreach \x in {1,2,3} {\node at (\x,5) {};}
			\draw[<->] (1,1) -- (7,1) ;
			\draw[<->] (1,2) -- (6,2) ;
			\draw[<->] (1,3) -- (4,3) ;
			\draw[<->] (1,4) -- (1,5) ;
			\draw[<->] (2,4) -- (2,5) ;
			\draw[<->] (3,4) -- (3,5) ;
			\end{tikzpicture}
			&\qquad \begin{tikzpicture}[baseline=-25pt,shorten >=-4pt,shorten <= - 4pt,
			scale=.4,yscale=-1,every node/.style={minimum size=.4cm,inner sep=0mm,draw,gray,rectangle}]
			\foreach \x in {1,2,3,4,5,6,7} {\node at (\x,1) {};}
			\foreach \x in {1,2,3,4,5,6} {\node at (\x,2) {};}
			\foreach \x in {1,2,3,4} {\node at (\x,3) {};}
			\foreach \x in {1,2,3} {\node at (\x,4) {};}
			\foreach \x in {1,2,3} {\node at (\x,5) {};}
			\draw[<->] (1,1) -- (7,1) ;
			\draw[<->] (1,2) -- (6,2) ;
			\draw[<->] (1,3) -- (1,5) ;
			\draw[<->] (2,3) -- (2,5) ;
			\draw[<->] (3,3) -- (3,5) ;
			\draw[<->] (3.99,3) -- (4.01,3);
			\end{tikzpicture}
			&\qquad \begin{tikzpicture}[baseline=-25pt,shorten >=-4pt,shorten <= - 4pt,scale=.4,yscale=-1,every node/.style={minimum size=.4cm,inner sep=0mm,draw,gray,rectangle}]
			\foreach \x in {1,2,3,4,5,6,7} {\node at (\x,1) {};}
			\foreach \x in {1,2,3,4,5,6} {\node at (\x,2) {};}
			\foreach \x in {1,2,3,4} {\node at (\x,3) {};}
			\foreach \x in {1,2,3} {\node at (\x,4) {};}
			\foreach \x in {1,2,3} {\node at (\x,5) {};}
			\draw[<->] (1,1) -- (7,1) ;
			\draw[<->] (1,2) -- (1,5) ;
			\draw[<->] (2,2) -- (6,2) ;
			\draw[<->] (2,3) -- (2,5) ;
			\draw[<->] (3,3) -- (4,3) ;
			\draw[<->] (3,4) -- (3,5) ;
			\end{tikzpicture}\\
			& \lambda^{\bs s_0}_{[\bs s_0]}=(7,6,4,2,2,2) &\lambda^{\bs s_1}_{[\bs s_1]}=(7,6,3,3,3,1) & \lambda^{\bs s_2}_{[\bs s_2]}=(7,4,5,3,2,2)\\
			& v^{\bs s_0}_{\lambda^{\bs s_0}}=v_\lambda  ,          & v^{\bs s_1}_{\lambda^{\bs s_1}}=e_{63}e_{53}e_{43}v_\lambda ,& v^{\bs s_2}_{\lambda^{\bs s_2}}=e_{53}e_{42}e_{43}v_\lambda.
			\nn
		\end{matrix} \ee
	\end{exmp}
	Another way to find $\lambda^{\bs s}_{[\bs s]}$ from $\la$ is given below in Theorem \ref{change of T}.

	Define the \emph{$\bs s$-Weyl weight} \[\rho^{\bs s}=\frac{1}{2}\sum_{\substack{\alpha\in\Phi^+_{\bs s}\\\alpha\mbox{ is even}}}\alpha-\frac{1}{2}\sum_{\substack{\beta\in\Phi^+_{\bs s}\\\beta\mbox{ is odd}}}\beta.\] 
	
	A weight $\lambda$ is called \emph{typical} if $(\lambda+\rho^{\bs s_0},\alpha)\neq0$, for any odd root $\alpha$. Otherwise $\lambda$ is called \emph{atypical}. The module $L(\lambda)$ is \emph{typical} if $\lambda$ is typical and \emph{atypical} otherwise. If $\lambda$ is a polynomial weight, then $\la$ is typical if and only if $\la(e_{MM})\geq N$. Let $\mu=(\mu_1,\mu_2,\dots)$
	be the $(M|N)$-hook partition that parametrizes $L(\lambda)$. Then $L(\lambda)$ is typical if and only if $\mu_M\geq N$. In Example \ref{ex1}, all weights are typical.
	
	\section{Rational pseudodifferential operators and flag varieties}\label{sec: rpdos}
	We establish some generalities about ratios of differential operators.
	
	\subsection{Rational pseudodifferential operators}
	We recall some results from \cite{CDSK12} and \cite{CDSK12b}.
	
	Let $\mc K$ be a differential field of characteristic zero, with the derivation $\partial$. The main example for this paper is the field of complex-valued rational functions $\mc K=\C(x)$.
	
	Consider the division ring of \emph{pseudodifferential operators} $\mc K\left((\partial^{-1})\right)$. An element $A\in\mc K\left((\partial^{-1})\right)$ has the form 
	\[A=\sum_{j=-\infty}^{M}a_{j}\partial^{j},\mbox{ }a_j\in\mc K,\mbox{ }M\in\Z.\] 
	One says that $A$ has \emph{order} $M$, $\ord A=M$, if $a_M\neq0$. One says that $A$ is \emph{monic} if $a_M=1$.
	
	We have the following relations in $\mc K\left((\partial^{-1})\right)$:
	\begin{align*}
	\partial\partial^{-1}=\partial^{-1}\partial=1,&\\
	\partial^{r}a=\sum_{j=0}^{\infty}\binom{r}{j}a^{(j)}\partial^{r-j},&\;a\in\mc K,\;r\in\Z,
	\end{align*} where $a^{(j)}$ is the $j$-th derivative of $a$ and $a^{(0)}=a$. 
	
	All nonzero elements in $\mc K\left((\partial^{-1})\right)$ are invertible. The inverse of $A$ is given by
	\[A^{-1}=\partial^{-M}\sum_{r=0}^{\infty}\Big(-\sum_{j=-\infty}^{-1}a_{M}^{-1}a_{j+M}\partial^{j}\Big)^{r}a_{M}^{-1}.\]
	
	The algebra of differential operators $\mc K[\partial]$ is a subring of $\mc K\left((\partial^{-1})\right)$. 
	
	Let $D\in\mc K[\partial]$ be a monic differential operator.  The differential operator $D$ is called \emph{completely factorable over $\mc K$} if $D=d_1\dots d_M$, where $d_i=\partial-a_i$, $a_i\in\mc K$, $i=1,\dots,M$. 
	
	Denote $\{u\in\mc K\;|\;Du=0\}$ by $\ker D$. Clearly, if $\dim\left(\ker D\right)=\ord D$, then $D$ is completely factorable over $\mc K$; see also Section \ref{factorization sec}.
	
	\medskip

	The division subring $\mc K(\partial)$ of $\mc K\left((\partial^{-1})\right)$, generated by $\mc K[\partial]$, is called the \emph{division ring of rational pseudodifferential operators} and elements in $\mc K(\partial)$ are called \emph{rational pseudodifferential operators}.

	Let $R$ be a rational pseudodifferential operator. If we can write $R=D_{\bar 0}D_{\bar 1}^{-1}$ for some $D_{\bar 0},D_{\bar 1}\in\mc K[\partial]$, then this is called a \emph{fractional factorization} of $R$. A fractional factorization $R=D_{\bar 0}D_{\bar 1}^{-1}$ is called \emph{minimal} if $D_{\bar 1}$ is monic and has the minimal possible order. 
	
	\begin{prop}{\cite{CDSK12b}}\label{CDSK12b} Let $R\in\mc K(\partial)$ be a rational pseudodifferential operator. Then the following is true.
		\begin{enumerate}
		\item There exists a unique minimal fractional factorization of $R$.
		\item Let $R=D_{\bar 0}D_{\bar 1}^{-1}$ be the minimal fractional factorization. If $R=\widetilde{D}_{\bar 0}\widetilde{D}_{\bar 1}^{-1}$ is a fractional factorization, then there exists $D\in\mc K[\partial]$ such that $\widetilde{D}_{\bar 0}=D_{\bar 0}D$ and $\widetilde{D}_{\bar 1}=D_{\bar 1}D$.
		\item Let $R=D_{\bar 0}D_{\bar 1}^{-1}$ be a fractional factorization such that $\dim\left(\ker D_{\bar 0}\right)=\ord D_{\bar 0} $ and $\dim\left(\ker D_{\bar 1} \right)=\ord D_{\bar 1} $. Then $R=D_{\bar 0}D_{\bar 1}^{-1}$ is the minimal fractional factorization of $R$ if and only if $\ker D_{\bar 0}\cap\ker D_{\bar 1}=0$.
		\end{enumerate}\qed
	\end{prop}

	We call $R$ an \emph{$(M|N)$-rational pseudodifferential operator} if for the minimal fractional factorization
	$R=D_{\bar 0}D_{\bar 1}^{-1}$ we have $\ord(D_{\bar 0})=M$ and $\ord(D_{\bar 1})=N$.

	Let $R$ be a monic $(M|N)$-rational pseudodifferential operator. Let $\bs s\in S_{M|N}$ be a parity sequence. The form $R=d_1^{s_1}\dots d_{M+N}^{s_{M+N}}$, where $d_i=\partial-a_i$, $a_i\in\mc K$, $i=1,\dots,M+N$, is called the \emph{complete factorization with the parity sequence $\bs s$}. We denote the set of all complete factorizations of $R$ by $\mc F(R)$ and the set of all complete factorizations of $R$ with parity sequence $\bs s$ by $\mc F^{\bs s}(R)$.

	Let  $R_1=(\partial-a)(\partial-b)^{-1}$ and $R_2=(\partial-c)^{-1}(\partial-d)$ be two $(1|1)$-rational pseudodifferential operators. Here $a,b,c,d\in\mc K$, $a\neq b$, and $c\neq d$. Then  $R_1=R_2$ if and only if
	\begin{equation}\label{Ore}
	\begin{matrix}
		& \begin{cases}
		 c=b+\ln'(a-b),\\
		 d=a+\ln'(a-b),
		 \end{cases} 
		& \mbox{ or equivalently }
		& \begin{cases}
		a=d-\ln'(c-d),\\
		b=c-\ln'(c-d),
		\end{cases}
	\end{matrix}
	\end{equation}
	where $\ln'(f)=f'/f$ stands for the logarithmic derivative.
	
	Let $R$ be an $(M|N)$-rational pseudodifferential operator.
	Let $R=d_1^{s_1}\dots d_{M+N}^{s_{M+N}}$,  $d_i=\partial-a_i$, be a complete factorization.
	Suppose $s_i\neq s_{i+1}$. Then $d_i\neq d_{i+1}$. We use equation \eqref{Ore} to construct $\widetilde{d}_i$ and $\widetilde{d}_{i+1}$ such that $d_i^{s_i}d_{i+1}^{s_{i+1}}=\widetilde{d}_i^{\,s_{i+1}}\widetilde{d}_{i+1}^{\,s_i}$. That gives a complete factorization of $R=d_1^{s_1}\dots \widetilde{d}_i^{\,s_{i+1}}\widetilde{d}_{i+1}^{\,s_i}\dots d_{M+N}^{s_{M+N}}$ with the new parity sequence $\widetilde{\bs s}=\bs s^{[i]}=(s_1,\dots,s_{i+1},s_i,\dots,s_{M+N})$.  
	
	Repeating this procedure, we obtain a canonical identification of the set $\mc F^{\bs s}(R)$ of complete factorizations of $R$ with parity sequence $\bs s$ with the set $\mc F^{\bs s_0}(R)$ of complete factorizations of $R$ with parity sequence $\bs s_0$.

	\subsection{Complete factorizations of rational pseudodifferential operators and flag varieties}\label{factorization sec}
	Let $W=W_{\bar{0}}\bigoplus W_{\bar{1}}$ be a vector superspace such that $\dim(W_{\bar{0}})=M$ and $\dim(W_{\bar{1}})=N$. A full flag in $W$ is a chain of subspaces $\mc F=\{F_1\subset F_2\subset\dots\subset F_{M+N}=W\}$ such that $\dim F_i=i$. 
	 Any basis  $\{w_1,\dots,w_{M+N}\}$ of $W$ generates a full flag by the rule $F_i={\rm span}(w_1,\dots,w_i)$.
(By basis, we mean always ordered basis.)
	A full flag is called a \emph{full superflag} if it is generated by a homogeneous basis. We denote by $\mc F(W)$ the set of all full superflags. 
	
	If $M=0$ or $N=0$, then every full flag is a full superflag. Thus, in this case $\mc F(W)$ is the usual flag variety.
	
	To a given homogeneous basis $\{w_1,\dots,w_{M+N}\}$ of $W$, we associate a parity sequence $\bs s\in S_{M|N}$ by the rule $s_i=(-1)^{|w_i|}$, $i=1,\dots,M+N$. We say a full superflag $\mc F$ has parity sequence $\bs s$ if it is generated by a homogenous basis associated to $\bs s$. We denote by $\mc F^{\bs s}(W)$ the set of all full superflags of parity $\bs s$.
	
The following lemma is obvious. 
	
	\begin{lem} We have
	\[
\mc F(W)=	\bigsqcup_{\bm{s}\in S_{M|N}} \mc F^{\bs s}(W), \qquad 
\mc F^{\bs s}(W)=\mc{F}\left(W_{\bar{0}}\right)\times \mc{F}\left(W_{\bar{1}}\right).\] \qed
	\end{lem}
	
Let $R$ be a monic $(M|N)$-rational pseudodifferential operator over $\mc K$. Let $R=D_{\bar 0}D_{\bar 1}^{-1}$ be the minimal fractional factorization of $R$. Assume that $\dim\left(\ker D_{\bar 0} \right)=M$, and $\dim\left(\ker D_{\bar 1} \right)=N$. 
	
	Let $V=W_{\bar 0}=\ker D_{\bar 0} $, $U=W_{\bar 1}=\ker D_{\bar 1} $, $W=W_{\bar 0}\oplus W_{\bar 1}$.
	
	Given a basis $\{v_1,\dots, v_M\}$ of $V$, a basis $\{u_1,\dots, u_N\}$ of $U$, and a parity sequence $\bs s\in S_{M|N}$, define a homogeneous basis $\{w_1,\dots, w_{M+N}\}$ of $W$ by the rule $w_i=v_{\bs s^+_i+1}$ if $s_i=1$ and 
	$w_i=u_{\bs s^-_i+1}$ if $s_i=-1$. Conversely, any homogeneous basis of $W$ gives a basis of $V$, a basis of $U$, and a parity sequence $\bs s$.
	
	\begin{exmp}
	If $\bs s=(1,-1,-1,1,1,-1,1,-1)$, then $\{w_1,\dots,w_8\}=\{v_4,u_1,u_2,v_3,v_2,u_3,v_1,u_4\}$.
	\end{exmp}
	
	Given a basis $\{v_1,\dots, v_M\}$ of $V$, a basis $\{u_1,\dots, u_N\}$ of $U$, and a parity sequence $\bs s\in S_{M|N}$, define  
	$d_i=d_i(\bs s,\{v_1,\dots, v_M\},\{u_1,\dots, u_N\})=\partial-a_i$ where
	\begin{align}
	 &a_i= \ln'  \frac{\Wr(v_1,v_2,\dots,v_{\bs s_i^++1},u_1,u_2,\dots,u_{\bs s_i^-})}{\Wr(v_1,v_2,\dots,v_{\bs s_i^+},u_1,u_2,\dots,u_{\bs s_i^-})} \qquad {\rm if}\ s_i=1, \label{wronskian of rdo1}\\ &a_i= \ln'  \frac{\Wr(v_1,v_2,\dots,v_{\bs s_i^+},u_1,u_2,\dots,u_{\bs s_i^-+1})}{\Wr(v_1,v_2,\dots,v_{\bs s_i^+},u_1,u_2,\dots,u_{\bs s_i^-})}  \qquad {\rm if}\ s_i=-1,\label{wronskian of rdo2}
\end{align}
where the Wronskian is given by the standard formula
$$
\Wr(f_1,\dots,f_r)=\det \left( f_j^{(i-1)}  \right)_{i,j=1}^r.
$$
If two bases $\{v_1,\dots, v_M\}$, $\{\wt{v}_1,\dots, \wt{v}_M\}$ generate the same full flag of $V$ and two bases $\{u_1,\dots, u_N\}$, $\{\wt{u}_1,\dots, \wt{u}_N\}$ generate the same full flag of $U$, then the coefficients $a_i$ computed from $v_j,u_j$ and from $\wt{v}_j,\wt{u}_j$ coincide.

\begin{prop}\label{factor}
We have a complete decomposition of $R$ with parity $\bs s$: $R=d_1^{s_1}\dots d_{M+N}^{s_{M+N}}$.
\end{prop}
\begin{proof}
If $\bs s=\bs s_0$ is standard, then the statement of the proposition is well known: see for example the Appendix in \cite{MVpopulations}.

Let $\bs s$ and $\widetilde {\bs s}$ differ only in positions $i, i+1$: $s_j=\wt {s}_j$ for $j\neq i, i+1$ and $s_i=-s_{i+1}=-\wt {s}_i=\wt {s}_{i+1}$.
Then we have $d_j=\wt {d}_j$ for $j\neq i, i+1$. In addition $d_i^{s_i}d_{i+1}^{s_{i+1}}=\widetilde {d}_i^{\ \wt{s}_i}\wt{d}_{i+1}^{\ \wt{s}_{i+1}}$ follows from the Wronski identity
\begin{align*}
&\Wr\left(\Wr(v_1,v_2,\dots,v_{\bs s_i^++1},u_1,u_2,\dots,u_{\bs s_i^-}),\Wr(v_1,v_2,\dots,v_{\bs s_i^+},u_1,u_2,\dots,u_{\bs s_i^-+1})\right) \\
&=
\Wr(v_1,v_2,\dots,v_{\bs s_i^++1},u_1,u_2,\dots,u_{\bs s_i^-+1})\Wr(v_1,v_2,\dots,v_{\bs s_i^+},u_1,u_2,\dots,u_{\bs s_i^-}).
\end{align*}
\end{proof}

We identify full superflags in $W$ with complete factorizations of $R$. 
Namely, by Proposition \ref{factor} we have a map: $\rho: \mc F(W) \to \mc F(R)$ and $\rho^{\bs s}: \mc F^{\bs s}(W) \to \mc F^{\bs s}(R)$.
\begin{prop}\label{flag and factorization}
The maps $\rho$, $\rho^{\bs s}$ are bijections.
\end{prop}
\begin{proof}
Clearly, $\rho^{\bs s_0}$ is a bijiection.  We have a canonical bijection between $\mc F^{\bs s}(W)$ and $\mc F^{\bs s_0}(W)$.
We have a canonical bijection between $\mc F^{\bs s}(R)$ and $\mc F^{\bs s_0}(R)$. These two bijections are compatible with $\rho^{\bs s}$ and
$\rho^{\bs s_0}$. The proposition follows.
\end{proof}

	\section{Bethe ansatz}\label{sec: ba}
	We recall some facts about the Gaudin model associated to $\gl_{M|N}$; see, for example, \cite{MVY14}.
	\subsection{Gaudin Hamiltonians}
	Let ($V_1$,\dots,$V_n$) be a sequence of $\gl_{M|N}$ modules. Let $\bs z=(z_1,\dots,z_n)$ be a sequence of pairwise distinct complex numbers. Consider the tensor product $V=\bigotimes_{k=1}^n V_k$.  The \emph{Gaudin Hamiltonians} $\HH_r\in\End(V)$, $r=1,\dots,n$, are given by \[\HH_r=\sum_{\substack{k=1\\k\neq r}}^{n}\frac{\sum_{a,b=1}^{M+N}e_{ab}^{(r)}e_{ba}^{(k)}(-1)^{|b|}}{z_r-z_k}, \]
	where $e_{ab}^{(k)}=\underset{k-1}{\underbrace{ 1\otimes \dots \otimes 1}}\otimes e_{ab} \otimes \underset{n-k}{\underbrace{ 1\otimes \dots \otimes 1}}$, $k=1,\dots,n$.
	
The proof of the following properties (which are well-known in the case of $\gl_M$) can be found in \cite{MVY14}.
	\begin{lem}
		We have:
		\begin{enumerate}
			\item the Gaudin Hamiltonians mutually commute, $[\HH_r,\HH_k]=0$, for all $r,k$;
			\item the Gaudin Hamiltonians commute with the diagonal $\gl_{M|N}$ action, $[\HH_k,X]=0$, for all $k$ and all $X\in \gl_{M|N}$;
			\item the sum of the Gaudin Hamiltonians is zero, $\sum_{k=1}^{n}\HH_k=0$;
			\item if $V_k,k=1,\dots,n$, are polynomial modules, then for generic $z_k,k=1,\dots,n$, the Gaudin Hamiltonians are diagonalizable;
			\item if $V_k,k=1,\dots,n$, are vector representations, then the joint spectrum of the Gaudin Hamiltonians is simple for generic $\bs z$.
		\end{enumerate}
		\qed
	\end{lem}
	
	\subsection{Bethe ansatz equation} \label{sec: BAE}
	We fix a parity sequence $\bs s\in S_{M|N}$, a sequence $\bs \lambda=(\lambda^{(1)},\dots,\lambda^{(n)})$ of $\gl_{M|N}$ weights, and a sequence $\bs z=(z_1,\dots,z_n)$ of pairwise distinct complex numbers. We call $(\lambda^{(k)})^{\bm{s}}$ the \emph{weight at the point $z_k$ with respect to $\bm{s}$}.
	
	Let $\bs l=(l_1,\dots,l_{M+N-1})$ be a sequence of non-negative integers. Define $l=\sum_{i=1}^{M+N-1}l_i$. Let $c:\{1,\dots,l\}\rightarrow\{1,\dots,M+N-1\}$ be the \emph{colour function}, \[c(j)=r,\;\mbox{if}\;\sum_{i=1}^{r-1}l_i<j\leq\sum_{i=1}^{r}l_i.\] Let $\bm{t}=(t_1,\dots,t_l)$ be a collection of variables.  We say that $t_j$ has \emph{colour $c(j)$}. Define the \emph{weight at $\infty$ with respect to $\bm{s}$, $\bm{\lambda}$, and $\bm{l}$} by 
	\[
	\lambda^{(\bs s,\infty)}=\sum_{k=1}^{n}(\lambda^{(k)})^{\bm{s}}-\sum_{i=1}^{M+N-1}\alpha^{\bs s}_{i} l_i.
	\]
	
	The \emph{Bethe ansatz equation} (BAE) associated to $\bs s$, $\bs z$, $\bm\lambda$, and $\bm{l}$, is a system of algebraic equations on variables $\bs t$:
	\begin{equation}\label{BAE}
	-\sum_{k=1}^n\frac{((\lambda^{(k)})^{\bm{s}},\alpha^{\bs s}_{c(j)})}{t_{j}-z_k}+\sum_{\substack{r=1\\r\neq j}}^{l}\frac{(\alpha^{\bs s}_{c(r)},\alpha^{\bs s}_{c(j)})}{t_j-t_r}=0,\;j=1,\dots,l.
	\end{equation}
	The BAE is a system of equations for $\bs t$ and we call the single equation \eqref{BAE} \emph{the Bethe ansatz equation for $\bs t$ related to $t_j$}.
	
	Note that if $\bs t$ is a solution of the BAE and $(\alpha^{\bs s}_{c(r)},\alpha^{\bs s}_{c(j)})\neq 0$ for some $j\neq r$, then $t_{j}\neq t_{r}$. Also if $((\lambda^{(k)})^{\bm{s}},\alpha^{\bs s}_{c(j)})\neq 0$ for some $k$ and $j$, then $t_j\neq z_k$. 

In addition, we impose the following condition. Suppose $(\alpha^{\bs s}_{i},\alpha^{\bs s}_{i})=0$. Choose $j$ such that $c(j)=i$ and
consider the equation related to $t_j$ as an equation for one variable when all variables $t_r$ with $c(r)\neq i$ are fixed. This equation does not depend on the choice of $j$. Suppose $t$ is a solution of this equation of multiplicity $a$. Then we require that the number of $t_j$ such that $c(j)=i$ and $t_j=t$ is at most $a$. This condition will be important in what follows; cf. especially Lemma \ref{gl1|1 rplemma}, Theorem \ref{GRP}, and Conjecture \ref{conj1}.
	
	The group $\mathfrak{S}_{\bs l}=\mathfrak{S}_{l_1}\times\dots\times\mathfrak{S}_{l_{M+N-1}}$ acts on $\bs t$ by permuting the variables of the same colour. 

	We do not distinguish between solutions of the BAE in the same $\mathfrak{S}_{\bs l}$-orbit.

	\subsection{Weight function} Let $\lambda^{(k)}$, $k=1,\dots,n$, be polynomial $\gl_{M|N}$ weights. Let $v_k^{\bs s}=v^{\bs s}_{(\lambda^{(k)})^{\bm{s}}}$ be an $\bs s$-highest weight vector in the irreducible $\gl_{M|N}$ module $L(\lambda^{(k)})$. Consider the tensor product $L(\bm\lambda)=\bigotimes_{k=1}^n L(\lambda^{(k)})$. The \emph{weight function} is a vector $w^{\bs s}(\bs z,\bs t)$ in $L(\bm\lambda)$ depending on parameters $\bs z=(z_1,\dots,z_n)$ and variables $\bs t=(t_1,\dots,t_l)$. The weight function $w^{\bs s}(\bs z,\bs t)$ is constructed as follows (see \cite{MVY14}).
	
	Let an \emph{ordered partition} of $\{1,\dots,l\}$ into $n$ parts be a sequence $\bs I=(i^1_1,\dots,i^1_{p_1};\dots;i^n_1,\dots,i^n_{p_n})$, where $p_1+\dots+p_n=l$ and $\bs I$ is a permutation of $(1,\dots,l)$. Let $P(l,n)$ be the set of all such ordered partitions.
	
	Denote $F^{\bs s}_{c(r)}=e^{\bs s}_{c(r)+1,c(r)}$. To each ordered partition $\bs I\in P(l,n)$,  associate  a vector $F^{\bs s}_{\bs I}v\in L(\bm\lambda)$ and a rational function $w_{\bs I}(\bs z, \bs t)$,
	\begin{align*}
	& F^{\bs s}_{\bs I}v=F^{\bs s}_{c(i^1_1)}\dots F^{\bs s}_{c(i^1_{p_1})}v^{\bs s}_1\otimes\dots\otimes F^{\bs s}_{c(i^n_1)}\dots F^{\bs s}_{c(i^n_{p_n})}v^{\bs s}_{n},\\
	& w_{\bs I}(\bs z,\bs t)=w_{\{i^1_1,\dots,i^1_{p_1}\}}(z_1,\bs t)\dots w_{\{i^n_1,\dots,i^n_{p_n}\}}(z_n,\bs t),
	\end{align*}
	where for $\{i_1,\dots,i_r\}\subset\{1,\dots,l\}$, \[w_{\{i_1,\dots,i_r\}}(z,\bs t)=\frac{1}{(t_{i_1}-t_{i_2})\dots(t_{i_{r-1}}-t_{i_r})(t_{i_r}-z)}.\]
	Define \[(-1)^{|\bs I|}=\prod_{r=1}^l\prod_{\substack{j>r\\{\bs I}(j)<{\bs I}(r)}}(-1)^{|F^{\bs s}_{c(r)}|\cdot|F^{\bs s}_{c(j)}|}.\] 
	
	Then the weight function $w^{\bs s}(\bs z,\bs t)$ is  \begin{equation}
	w^{\bs s}(\bs z,\bs t)=\sum_{\bs I\in P(l,n)}(-1)^{|\bs I|}w_{\bs I}(\bs z,\bs t)F^{\bs s}_{\bs I}v.
	\end{equation}
	
	We have the following theorem.
	\begin{thm}{\cite{MVY14}}\label{MVY14}
		If $\bm{\lambda}$ is a sequence of polynomial weights and $\bs t$ is a solution of the BAE associated to $\bs s$, $\bs z$, $\bm\lambda$, and $\bm{l}$, then the vector $w^{\bs s}(\bs z,\bs t)\in L(\bs \lambda)$ is a joint eigenvector of the Gaudin Hamiltonians, $\HH_k w^{\bs s}(\bs z,\bs t)=E_kw^{\bs s}(\bs z,\bs t)$, $k=1,\dots,n$, where the eigenvalues $E_k$ are given by
		\begin{equation}\label{hamiltonian}
		E_k=\sum_{\substack{r=1\\r\neq k}}^{n}\frac{\left((\lambda^{(k)})^{\bm{s}},(\lambda^{(r)})^{\bm{s}}\right)}{z_k-z_r}+\sum_{j=1}^l\frac{\left((\lambda^{(k)})^{\bm{s}},\alpha^{\bs s}_{c(j)}\right)}{t_j-z_k}.
		\end{equation}
		Moreover, the vector $w^{\bs s}(\bs z,\bs t)$ belongs to $\left(L(\bm\lambda)\right)^{\textup{$\bs s$ing}}_{\lambda^{(\bs s,\infty)}}$. 
		\qed
	\end{thm}
	If $\bs t$ is a solution of the BAE associated to $\bs s$, $\bs z$, $\bm\lambda$, and $\bm{l}$, then the value of the weight function $w^{\bs s}(\bs z,\bs t)$ is called a \emph{Bethe vector}.

	\subsection{Polynomials representing solutions of the BAE}
	
	Fix a parity sequence $\bs s\in S_{M|N}$. Let $\bm\lambda=(\lambda^{(1)},\dots,\lambda^{(n)})$ be a sequence of polynomial $\gl_{M|N}$ weights. Let $\bs z=(z_1,\dots,z_n)$ be a sequence of pairwise distinct complex numbers.
	
	Define a sequence of polynomials $\bm{T}^{\bm{s}}=(T_1^{\bm{s}},\dots,T_{M+N}^{\bm{s}})$ associated to $\bm{s}$, $\bm{\lambda}$ and $\bm{z}$, 
	\begin{equation}\label{T fun}
	T_i^{\bm{s}}(x)=\prod_{k=1}^{n}(x-z_k)^{(\lambda^{(k)})^{\bm{s}}(e^{\bs s}_{ii})},\;i=1,\dots,M+N.
	\end{equation}
	Note that  $T^{\bs s}_i(T^{\bm{s}}_{i+1})^{-s_is_{i+1}}$ is a polynomial for all $i=1,\dots,M+N$.

	Let $\bs l=(l_1,\dots,l_{M+N-1})$ be a sequence of non-negative integers. Let $\bm{t}=(t_1,\dots,t_l)$ be a solution of the BAE associated to $\bm{s}$, $\bm{z}$, $\bm{\lambda}$, and $\bm{l}$. Define a sequence of polynomials $\bm{y}=(y_1,\dots,y_{M+N-1})$ by 
\be y_i(x)=\prod_{j,\ c(j)=i}(x-t_j),\;i=1,\dots,M+N-1\label{y fun}.\ee	 
We say the \emph{sequence of polynomials $\bm{y}$ represents $\bm{t}$}.
	
We consider each polynomial $y_i(x)$ up to a multiplication by a non-zero number. We also do not consider zero polynomials $y_i(x)$.
Thus, the sequence $\bs y$ defines a point in the direct product $\mathbb P(\C[x])^{M+N-1}$ of $M+N-1$ copies of the projective space associated to the vector space of polynomials in $x$. We also have $\deg y_i=l_i$. 
	
	\medskip
	
	A sequence of polynomials $\bs y$ is \emph{generic with respect to $\bs s$, $\bm\lambda$, and $\bs z$}, if it satisfies the following conditions:
	\begin{enumerate}
		\item if $s_is_{i+1}=1$, then $y_i(x)$ has only simple roots;
		\item if $(\alpha^{\bs s}_{i},\alpha^{\bs s}_{j})\neq 0$ and $i\neq j$, then $y_i(x)$ and $y_j(x)$ have no common roots;
		\item all roots of $y_i(x)$ are different from the roots of $T^{\bs s}_i(x)(T^{\bs s}_{i+1}(x))^{-s_is_{i+1}}$.
	\end{enumerate}
	If $\bs y$ represents a solution of the BAE associated to $\bs s$, $\bs z$, $\bm{\lambda}$, and $\bs l$, then $\bs y$ is generic with respect to $\bs s$, $\bm\lambda$, and $\bs z$.
	
	
	\section{Reproduction procedure for $\gl_2$ and $\gl_{1|1}$}\label{sec: rep}
	We recall the reproduction procedure for $\gl_2$, see \cite{MVpopulations}, and define its analogue for $\gl_{1|1}$.
	
	\subsection{Reproduction procedure for $\gl_2$}
	Consider the case of $M=2$ and $N=0$. We write $\gl_{2|0}\cong\gl_{0|2}\cong\gl_2$. Let $\bm\lambda=(\lambda^{(1)},\dots,\lambda^{(n)})=\left((p_1,q_1),\dots,(p_n,q_n)\right)$ be a sequence of polynomial $\gl_2$ weights: $p_k, q_k \in\Z$,  $p_k\geq q_k\geq 0$, $k=1,\dots,n$.  Let $\bs z=(z_1,\dots,z_n)$ be a sequence of pairwise distinct complex numbers.  We have \[
	T_1=\prod_{k=1}^n(x-z_k)^{p_k},\qquad\qquad T_2=\prod_{k=1}^n(x-z_k)^{q_k}.
	\]
	Let $p=\deg T_1$ and $q=\deg T_2$.

	Let $l$ be a non-negative integer. Let $\bs t=(t_1,\dots,t_l)$ be a collection of variables. The Bethe ansatz equation associated to $\bm\lambda$, $\bs z$ and $l$, is given by \begin{equation}\label{sl2BAE}
	-\sum_{k=1}^n\frac{p_k-q_k}{t_j-z_k}+\sum_{\substack{r=1\\r\neq j}}^{l}\frac{2}{t_j-t_r}=0,\;j=1,\dots,l.
	\end{equation}

	One can reformulate the BAE \eqref{sl2BAE} and construct a family of new solutions of the BAE as follows. 
	\begin{lem}{\cite{MVpopulations}}\label{MVpopulations} Let $y$ be a degree $l$ polynomial generic with respect to $\bm{\lambda}$ and $\bs z$.
		\begin{enumerate}
			\item The polynomial $y\in\C[x]$ represents a solution of the BAE \eqref{sl2BAE} associated to $\bm\lambda$, $\bs z$ and $l$, if and only if there exists a polynomial $\widetilde{y}\in\C[x]$, such that 
			\begin{equation}\label{sl2wr}
			\Wr(y,\widetilde{y})=T_1T_2^{-1}.
			\end{equation}
			\item  If $\widetilde{y}$ is generic, then $\widetilde{y}$ represents a solution of the BAE associated to $\bm\lambda$, $\bs z$ and $\widetilde{l}$, where $\widetilde{l}=\deg \widetilde{y}$.
		\end{enumerate}		
		\qed
	\end{lem}
	Explicitly, the polynomial $\widetilde{y}$ in Lemma \ref{MVpopulations} is given by \begin{equation}\label{c1c2}
	\widetilde{y}(x)=c_1y(x)\int T_1(x)T_2^{-1}(x)y^{-2}(x)dx+c_2y(x),
	\end{equation} where $c_1$ is some non-zero complex number and $c_2\in\C$ is arbitrary. The BAE \eqref{sl2BAE} guarantees that the integrand has no residues and therefore $\widetilde{y}$ is a polynomial. All but finitely many $\widetilde{y}$ are generic with respect to $\bm{\lambda}$ and $\bs z$, and therefore represent solutions of the BAE \eqref{sl2BAE}. 
	
	Thus, from the polynomial $y$, we construct a family of polynomials $\widetilde{y}$. Following \cite{MVpopulations}, we call this construction the \emph{$\gl_2$ reproduction procedure}.
	
	Let $P_y$ be the closure of the set containing $y$ and all $\widetilde{y}$ in $\mathbb P(\C[x])$. We call $P_y$ the \emph{$\gl_2$ population originated at $y$}. The set $P_y$ is identified with the projective line $\C P^1$ with projective coordinates $(c_1:c_2)$.
	
	The weight at infinity associated to $\bs \lambda ,l$,  is $\lambda^{(\infty)}=(p-l,q+l)$. 
	Assume the weight $\lambda^{(\infty)}$ is dominant, meaning $2l\leq p-q$.
	Then the weight at infinity associated to $\bs \lambda ,\widetilde{l}$, is
	\[\widetilde{\lambda}^{(\infty)}=(p-\widetilde{l},q+\widetilde{l})=(q+l-1,p-l+1)=s\cdot \lambda^{(\infty)},\]
where $s\in\mathfrak S_2$ is the non-trivial $\gl_2$ Weyl group element, and the dot denotes the shifted action.
	
	Let $\widetilde y=\prod_{r=1}^{\widetilde{l}}(x- \widetilde{t}_r)$ and $\bm{\widetilde{t}}=(\widetilde{t}_1,\dots,\widetilde{t}_{\widetilde{l}})$. If $y$ is generic, then by Lemma \ref{MVpopulations}, $\widetilde {\bs t}$ is a solution of the BAE \eqref{sl2BAE}. By Theorem \ref{MVY14}, the value of the weight function $w(\bs z,\bm{\widetilde{t}})$ is a singular vector. However, $\widetilde{\lambda}^{(\infty)}$ is not dominant and therefore  $w(\bs z,\bm{\widetilde{t}})=0$ in $L(\bs \lambda)$. So, in a $\gl_2$ population only the unique smallest degree polynomial corresponds to an actual eigenvector in $L(\bs\lambda)$.
   
Consider formula \eqref{hamiltonian} for the eigenvalues $E_k$ of the Gaudin Hamiltonians. Clearly,
\[
\ln'{y}(z_k)=\ln'{\widetilde{y}}(z_k), \; k=1,\dots, n,
\]
which implies that the eigenvalues $E_k$ for the solution $\bs t$ of the BAE  are equal to those for the solution $\widetilde{\bs t}$.
That fact can be reformulated in the following form.
  
	Define a differential operator
	$$
	D(y)=\left(\partial - \ln'\frac{T_1}{y}\right)(\partial-\ln' T_2\,y).
	$$
	The operator $D(y)$ does not depend on a choice of polynomial $y$ in a population, $D(y)=D(\wt y)$.
	
	\subsection{Reproduction procedure for $\gl_{1|1}$}\label{Reproduction procedure for gl{1|1}} Consider the case of $M=N=1$.
	We have $ S_{1|1}=\{(1,-1),(-1,1)\}$. Let $\bs s$ and $\widetilde{\bs s}=\bs s^{[1]}$ be two different parity sequences. Let $\bm\lambda=(\lambda^{(1)},\dots,\lambda^{(n)})$ be a sequence of polynomial $\gl_{1|1}$ weights. For each $k=1,\dots,n$, let us write $(\lambda^{(k)})^{\bs s}_{[\bs s]}=(p_k,q_k)$, where $p_k,q_k\in\Z_{\geq 0}$ and if $p_k=0$ then $q_k=0$. Note that $\lambda^{(k)}$ is atypical if and only if it is zero, $p_k=q_k=0$, which happens if and only if $p_k+q_k = 0$. 
Let $\bs z=(z_1,\dots,z_n)$ be a sequence of pairwise distinct complex numbers.
	
	Let $$\widetilde{p}_k=\begin{cases} q_k+1 \ \ &{\rm if}\ \ p_k+q_k\neq 0,
	\\ 0\ &{\rm if}\ \ p_k+q_k=0,\end{cases} \qquad  \widetilde{q}_k=\begin{cases} p_k-1 \ \ &{\rm if}\ \ p_k+q_k\neq 0,
	\\ 0\ &{\rm if}\ \ p_k+q_k=0.\end{cases}  $$

	Equation \eqref{T fun} becomes
	\begin{align*}
	 &T_1^{\bm{s}}=\prod_{k=1}^n(x-z_k)^{p_k},\quad T_2^{\bm{s}}=\prod_{k=1}^n(x-z_k)^{q_k},\\
	 T_1^{\widetilde{\bm{s}}}=\prod_{\substack{k=1 \\ p_k+q_k\neq 0}}^n(x-z_k)^{q_k+1}&=\prod_{k=1}^n(x-z_k)^{\widetilde{p}_k}, \quad 
	T_2^{\widetilde{\bm{s}}}=\prod_{\substack{k=1 \\ p_k+q_k\neq 0}}^n(x-z_k)^{p_k-1}=\prod_{k=1}^n(x-z_k)^{\widetilde{q}_k}.
	\end{align*}
	Let $p=\deg T_1^{\bs s}$, $q=\deg T_2^{\bs s}$. Similarly, let $\widetilde{p}=\deg T_1^{\widetilde{\bs s}}$, $\widetilde{q}=\deg T_2^{\widetilde{\bs s}}$.
	
	Let $m=\#\{k\;|\; p_k+q_k\neq 0\}$ be the number of typical modules. Then $\widetilde{p}=q+m$ and $\widetilde{q}=p-m$.

	Let $l$ be a non-negative integer. Let $\bs t=(t_1,\dots,t_l)$ be a collection of variables. The Bethe ansatz equation associated to $\bs s$, $\bm\lambda$, $\bs z$, and $l$, takes the form:
	\be\label{gl(1|1)BAE}
	\sum_{k=1}^n\frac{p_k+q_k}{t_j-z_k}=0, \; j=1,\dots,l.
	\ee
	
	 The Bethe ansatz equation \eqref{gl(1|1)BAE} can be written in the form		
	\[\ln'\left(T_1^{\bm{s}}T_2^{\bm{s}}\right)(t_j)=0.\] 

	Note that $T_1^{\bm{s}}T_2^{\bm{s}}=T_1^{\widetilde{\bm{s}}}T_2^{\widetilde{\bm{s}}}$. Thus, in the case of $\gl_{1|1}$, the BAEs \eqref{gl(1|1)BAE} associated to $\bs s$ and $\widetilde{\bs s}$ coincide. 
	
	\medskip
	
	Define a map $\pi$ from non-zero rational functions $\C(x)$ to monic polynomials in $\C[x]$ with distinct roots. For any nonzero rational function $f(x)$, $\pi(f)(z)=0$ if and only if $f(z)=0$ or $(1/f)(z)=0$. 

	\begin{exmp}
	We have $\pi\left(x^5(x-1)^4(x-3)^{-1}(x+6)^{-2}\right)=x(x-1)(x-3)(x+6)$.
	\end{exmp}
	
	The polynomial $\pi(f)$ is the minimal monic denominator of the rational function $\ln'(f)$ of smallest possible degree.
		
		\medskip
		
	We call the sequence of polynomial $\gl_{1|1}$ weights $\bs \lambda$ \emph{typical} if at least one of the weights $\lambda^{(k)}$ is typical. 
	Then $\bs \lambda$ is typical if and only if $p+q\neq 0$. Also $\bs \lambda$ is not typical if and only if $T_1^{\bm{s}}T_2^{\bm{s}}=1$.
	
	We reformulate the BAE \eqref{gl(1|1)BAE} and construct a new solution as follows.
	\begin{lem}\label{gl1|1 rplemma} Let $y$ be a polynomial of degree $l$. Let $\bs \lambda$ be typical.
		\begin{enumerate}
			\item The polynomial $y$ represents a solution of the BAE \eqref{gl(1|1)BAE} associated to $\bs s$, $\bs z$, $\bm\lambda$, and $l$, if and only if there exists a polynomial $\widetilde{y}$, such that 
			\begin{equation}\label{gl1|1 rp}
			y\cdot\widetilde{y}=\ln'\left(T_1^{\bm{s}}T_2^{\bm{s}}\right)\pi(T^{\bs s}_1T^{\bs s}_2).
			\end{equation}
			\item  The polynomial $\widetilde{y}$ represents a solution of the BAE \eqref{gl(1|1)BAE} associated to $\widetilde{\bm{s}}$, $\bs z$, $\bm{\lambda}$, and $\widetilde{l}$, where $\widetilde{l}=\deg \widetilde{y}=m-1-l$.
			\qed
		\end{enumerate}
	\end{lem}
	From the polynomial $y$, we construct a unique polynomial $\widetilde{y}$. We call this construction the \emph{$\gl_{1|1}$ reproduction procedure}.
	
	Let $P_y$ be the set containing $y$ and $\widetilde{y}$. The set $P_y$ is called the \emph{$\gl_{1|1}$ population originated at $y$}.
	
	The weight at infinity associated to $\bs s, \bs \lambda$, and $l$ is $\lambda^{(\bs s,\infty)}_{[\bs s]}=(p-l,q+l)$. The weight at infinity associated to $\widetilde{\bs s}, \bs \lambda$ and $\widetilde{l}$ is $\widetilde{\lambda}^{(\widetilde{\bs s},\infty)}_{[\widetilde{\bm{s}}]}=(\widetilde{p}-\widetilde{l},\widetilde{q}+\widetilde{l})= (q+l+1,p-l-1)$. Thus we have $\lambda^{(\bs s,\infty)}=\widetilde{\lambda}^{(\widetilde{\bs s},\infty)}+\alpha^{\bs s}$.
	In particular, both $y$ and $\widetilde{y}$ correspond to actual eigenvectors of the Gaudin Hamiltonians.
	
	\begin{rem}\label{gl11rem}
	If $\bs \la$ is not typical, then all participating representations are one-dimensional and the situation is trivial. In particular, we have 
	$y(x)=1$. In this case we can define $\wt{y}=1$. We do not discuss this case any further. 
	\end{rem}
	
	\subsection{Motivation for $\gl_{1|1}$-reproduction procedure} 
	We show that in parallel to the $\gl_2$ reproduction procedure, the eigenvalues of the Gaudin Hamiltonians corresponding to polynomials in the same $\gl_{1|1}$ population are the same.
	
	Let $y=\prod_{r=1}^l(x-t_r)$,  $\widetilde{y}=\prod_{r=1}^{\widetilde{l}}(x-\widetilde{t}_r)$. Let 
	 $\bm{t}=(t_1,\dots,t_l)$, $\widetilde{\bm{t}}=(\widetilde{t}_1,\dots,\widetilde{t}_{\widetilde{l}})$. 
	
	Let $h_k=p_k+q_k$, $k=1,\dots, n$. Let $\mc N(T)$ be the monic polynomial proportional to $\ln'\left(T_1^{\bm{s}}T_2^{\bm{s}}\right)\pi(T_1^{\bm{s}}T_2^{\bm{s}})$.
	
	From Theorem \ref{MVY14}, we have $\HH_kw^{\bs s}(\bs z,\bs t)=E_kw^{\bs s}(\bs z,\bs t)$ and 
	$\HH_kw^{\widetilde{\bs s}}(\bs z,\widetilde{\bs t})=\wt{E}_kw^{\widetilde{\bs s}}(\bs z,\widetilde{\bs t})$, where
	\be\label{**}
	E_k=s_1\sum_{\substack{r=1\\r\neq k}}^{n}\frac{p_kp_r-q_kq_r}{z_k-z_r}+s_1\sum_{j=1}^{l}\frac{h_k}{t_j-z_k}, \qquad 
	\wt{E}_k=\widetilde{s}_1\sum_{\substack{r=1\\r\neq k}}^{n}\frac{\widetilde{p}_k\widetilde{p}_r-\widetilde{q}_k\widetilde{q}_r}{z_k-z_r}+\widetilde{s}_1\sum_{j=1}^{\wt l}\frac{h_k}{\widetilde{t}_j-z_k}.
	\ee

	\begin{lem}\label{gl11wf}
			The eigenvalues $E_k$ and $\widetilde{E}_k$, $k=1,\dots,n$, of the Gaudin Hamiltonians are the same.
	\end{lem}
	\begin{proof}
	Set $t_{l+r}=\widetilde t_r$, $r=1,\dots,\widetilde{l}$.
	
	If $p_k+q_k=0$, then $E_k=\widetilde{E}_k=0$. Without loss of generality, assume $p_k+q_k\neq 0$, $k=1,\dots, m$, $m>0$, and $p_k+q_k=0$, $k=m+1,\dots, n$, and consider $E_1-\widetilde{E}_1$. We have 
			\begin{equation}\label{*}
			s_1(E_1-\widetilde{E}_1)=\sum_{k=2}^{m}\frac{h_1+h_k}{z_1-z_k}+\sum_{r=1}^{m-1}\frac{h_1}{t_r-z_1}.
			\end{equation}
			The polynomial $\mc N(T)(x)$ is \[\mc N(T)(x)=\prod_{k=1}^{m-1}(x-t_k)=(h_1+\dots+h_m)^{-1}\sum_{k=1}^{m}h_k(x-z_1)\dots\widehat{(x-z_k)}\dots(x-z_m).\] Evaluate the function $\ln'(\mc N(T))$ at $z_1$ and we have \[\ln'(\mc N(T))(z_1)=\sum_{r=1}^{m-1}\frac{1}{z_1-t_r}=\sum_{k=2}^{m}\frac{h_1+h_k}{h_1(z_1-z_k)}.\] 
			Thus, the right-hand side of \eqref{*} is zero.
	\end{proof}

	\begin{cor}\label{e_{21}v}
			We have $e^{\bm{s}}_{21}w^{\bm{s}}(\bm{z},\bm{t})=cw^{\widetilde{\bm{s}}}(\bm{z},\widetilde{\bm{t}})$, for some non-zero constant $c$.
	\end{cor}
	\begin{proof}
	It follows from the results of \cite{MVY14} that for generic $\bs z$, the Gaudin Hamiltonians $\HH_k$ acting in $(L(\bs \lambda))^{{\bs s}{\rm ing}}=(\otimes_k L(\lambda^{k}))^{{\bs s}{\rm ing}}$	have joint simple spectrum. Moreover, for generic $\bs z$, $w^{\bm{s}}(\bm{z},\bm{t})\neq 0$ and
	$w^{\widetilde{\bm{s}}}(\bm{z},\widetilde{\bm{t}})\neq 0$. 
	
	Therefore, $w^{\bm{s}}(\bm{z},\bm{t})$ and $w^{\widetilde{\bm{s}}}(\bm{z},\widetilde{\bm{t}})$ belong to the same irreducible two-dimensional submodule of  $L(\bs \lambda)$. Moreover, their weights are related by $\lambda^{(\bs s,\infty)}=\wt{\lambda}^{(\widetilde{\bs s},\infty)}+\alpha^{\bs s}$. The corollary follows.
	\end{proof}
		
	Define a rational pseudodifferential operator:
	$$
	R^{\bs s}(y)=\left(\partial-s_1\ln'\frac{T^{\bs s}_1}{y}\right)^{s_1}\left(\partial-s_2\ln'(T^{\bs s}_2y)\right)^{s_2}.
	$$

	\begin{lem}\label{typlem} If $\bs\la$ is typical, then  $R^{\bs s}(y)$ is a $(1|1)$-rational pseudodifferential operator. If $\bs\la$ is not typical, then $R^{\bs s}(y)=1$.
	
	Let $\bs\la$ be typical. The rational pseudodifferential operator does not depend on a choice of a polynomial in a population:
	$R^{\bs s}(y)=R^{\wt{\bs s}}(\wt{y})$. 
	\end{lem}
	\begin{proof}
	The lemma is proved by a direct computation.
	\end{proof}

	\section{Reproduction procedure for $\gl_{M|N}$}\label{sec: repglmn}
	We define the reproduction procedure and populations in the general case.
	
	\subsection{Reproduction procedure}\label{6.1}
	Let $\bm{s}\in S_{M|N}$ be a parity sequence. Let $\bm{\lambda}=(\lambda^{(1)},\dots,\lambda^{(n)})$ be a sequence of polynomial $\gl_{M|N}$ weights. Let $\bm{z}=(z_1,\dots,z_n)$ be a sequence of pairwise distinct complex numbers. Let $\bm{T}^{\bm{s}}$ be a sequence of polynomials associated to $\bm{s}$, $\bm{\lambda}$, and $\bm{z}$, see \eqref{T fun}. Denote $\pi\left(T^{\bs s}_i(T^{\bs s}_{i+1})^{-s_is_{i+1}}\right)$ by $\pi^{\bs s}_i$.
	
	 For $i\in \{1,\dots, M+N-1\}$, set 
$\bm{s}^{[i]}=(s_1,\dots,s_{i+1},s_{i},\dots,s_{M+N})$. 
	
	\begin{lem}\label{change in T} If $s_i=s_{i+1}$, then $\bm{T}^{\bm{s}^{[i]}}=\bm{T}^{\bm{s}}$ and if $s_i\neq s_{i+1}$, then 
	\[\bm{T}^{\bm{s}^{[i]}}=(T_1^{\bm{s}},\dots, T^{\bm{s}}_{i+1}\pi_i^{\bs s},T^{\bm{s}}_{i}(\pi_i^{\bs s})^{-1}, \dots,T^{\bm{s}}_{M+N}).\]
	\end{lem} 
\begin{proof} This follows from \eqref{si}. \end{proof}
	Let $\bm{l}=(l_1,\dots,l_{M+N-1})$ be a sequence of nonnegative integers.
	
	We reformulate the BAE \eqref{BAE} and construct a family of new solutions as follows. By convention, we set $y_0=y_{M+N}=1$. 
	\begin{thm}\label{GRP}  Let $\bm{y}=(y_1,\dots,y_{M+N-1})$ be a sequence of polynomials generic with respect to $\bm{s}$, $\bm{\lambda}$, and $\bm{z}$, such that $\deg y_k=l_k$, $k=1,\dots,M+N-1$.
		\begin{enumerate}
			\item The sequence $\bm{y}$ represents a solution of the BAE \eqref{BAE} associated to $\bm{s}$, $\bm{z}$, $\bm{\lambda}$, and $\bm{l}$, if and only if for each $i=1,\dots,M+N-1$, there exists a polynomial $\widetilde{y}_i$, such that
			\begin{align}
			\Wr\left(y_i,\widetilde{y}_i\right)=T^{\bm{s}}_i\left(T^{\bm{s}}_{i+1}\right)^{-1}y_{i-1}y_{i+1} &\quad\mbox{if}\;s_i=s_{i+1},\label{brp}\\
			y_i\,\widetilde{y}_i=\ln'\left(\frac{T^{\bm{s}}_iT^{\bm{s}}_{i+1}y_{i-1}}{y_{i+1}}\right)\pi^{\bs s}_iy_{i-1}y_{i+1} &\quad\mbox{if}\;s_i\neq s_{i+1}.\label{frp}
			\end{align}
			\item Let $i\in \{1,\dots, M+N-1\}$ be such that $\wt y_i\neq 0$. Then if $\bm{y}^{[i]}=(y_1,\dots,\widetilde{y}_i,\dots,y_{M+N-1})$ 
	 is  generic with respect to $\bm{s}^{[i]}$, $\bm{\lambda}$, and $\bm{z}$, then $\bm{y}^{[i]}$ represents a solution of the BAE associated to $\bm{s}^{[i]}$, $\bm{\lambda}$, $\bm{z}$, and $\bm{l}^{[i]}$, where
			$\bm{l}^{[i]}=(l_1,\dots,\widetilde{l}_i,\dots,l_{M+N-1})$, $\widetilde{l}_i=\deg \widetilde{y}_i$.
		\end{enumerate}
	\end{thm}

	\begin{proof} 
		Part {\it (1)} follows from Lemma \ref{MVpopulations} and Lemma \ref{gl1|1 rplemma}.
		
		We prove Part {\it (2)}. Let $y_r=\prod_{j=1}^{l_r}(x-t^{(r)}_j)$, $r=1,\dots,M+N-1$, and $\widetilde{y}_i= \prod_{j=1}^{\wt{l}_i}(x-\widetilde{t}^{\;(i)}_j)$. Let $\bs t=(t_j^{(r)})_{r=1,\dots, M+N-1}^{j=1,\dots,l_r}$ and $\wt{\bs t}=(\wt {t}_j^{\;(r)})_{r=1,\dots, M+N-1}^{j=1,\dots,\wt{l}_r}$, where we set $l_r=\wt{l}_r, t_j^{(r)}=\wt {t}_j^{\;(r)}$ if $r\neq i$. The tuple ${\bs t}$ satisfies the BAE associated to $\bm{s}$, $\bm{\lambda}$, $\bm{z}$, and $\bm{l}$. We prove the Bethe ansatz equation for $\wt{\bs t}$ associated to $\bm{s}^{[i]}$, $\bm{\lambda}$, $\bm{z}$, and $\bm{l}^{[i]}$. The BAE for $\wt{\bs t}$ related to $\widetilde{t}^{\;(i)}_j$  holds by Lemma \ref{MVpopulations} and Lemma \ref{gl1|1 rplemma}. The BAEs for $\wt {\bs t}$ and ${\bs t}$ related to $t^{(r)}_j$, $|r-i|>1$, are the same. We treat the non-trivial cases.

		Consider the case of $s_{i}=s_{i+1}$. Dividing \eqref{brp} by $y_i\wt{y}_i$ and evaluating at $x=t^{(i\pm1)}_j$, we obtain \[\sum_{a=1}^{l_i}\frac{1}{t^{(i\pm1)}_j-t^{(i)}_a}=\sum_{a=1}^{\widetilde{l}_i}\frac{1}{t^{(i\pm1)}_j-\widetilde{t}^{\;(i)}_a}.\] Thus, the BAE for $\wt{\bs t}$ related to $t^{(i\pm1)}_j$ follows from the BAE for $\bs t$ related to $t^{(i\pm1)}_j$.
			
		Consider the case of $s_{i}=-s_{i+1}=1$. The argument depends on $s_{i-1}$, $s_{i+2}$. 
		Consider for example the case of $s_{i-1}=-s_{i+2}=1$.
		
		We prove the BAE for $\wt{\bs t}$ related to $t^{(i-1)}_j$:
		\begin{equation}
			-\sum_{k=1}^{n}\frac{(\lambda^{(k)})^{\bm{s}}_{[\bs s],i-1}+(\lambda^{(k)})^{\bm{s}}_{[\bs  s],i+1}+\delta}{t^{(i-1)}_j-z_k}+\sum_{r=1}^{l_{i-2}}\frac{-1}{t^{(i-1)}_j-t^{(i-2)}_r}+\sum_{r=1}^{\widetilde{l}_{i}}\frac{1}{t^{(i-1)}_j-\widetilde{t}^{\;(i)}_r}=0,\label{t_i}
		\end{equation}
		where	$\delta=1$	if $(\lambda^{(k)})^{\bm{s}}_{[\bs s],i}+(\lambda^{(k)})^{\bm{s}}_{[\bs s],i+1}\neq0$ and $\delta=0$ otherwise.
		
		The BAE for $\bs t$ related to $t^{(i-1)}_j$ is 
		\begin{equation}
			-\sum_{k=1}^{n}\frac{(\lambda^{(k)})^{\bm{s}}_{[\bs s],i-1}-(\lambda^{(k)})^{\bm{s}}_{[\bs s,]i}}{t^{(i-1)}_j-z_k}+\sum_{r=1}^{l_{i-2}}\frac{-1}{t^{(i-1)}_j-t^{(i-2)}_r}+\sum_{r=1}^{l_{i}}\frac{-1}{t^{(i-1)}_j-t^{(i)}_r}+\sum_{\substack{r=1\\r\neq j}}^{l_{i-1}}\frac{2}{t^{(i-1)}_j-t^{(i-1)}_r}=0.\label{t_i-1}\end{equation}

		Take the logarithmic derivative of equation \eqref{frp} for $y_i$ and evaluate it at $t^{(i-1)}_j$. The left-hand side is 
		\begin{align*}
			& \ln'(y_i\widetilde{y}_i)\Big |_{x=t^{(i-1)}_j}=\sum_{r=1}^{l_{i}}\frac{1}{t^{(i-1)}_j-t^{(i)}_r}+\sum_{r=1}^{\widetilde{l}_{i}}\frac{1}{t^{(i-1)}_j-\widetilde{t}^{\;(i)}_r}&
		\end{align*}
		and the right-hand side is
		\begin{align*}
		\ln'\big(\ln' &\left(T^{\bm{s}}_iT^{\bm{s}}_{i+1}  y_{i-1}y_{i+1}^{-1}\right)\pi^{\bs s}_iy_{i-1}y_{i+1}\big) \Big|_{x=t^{(i-1)}_j}\\
		&\hspace{30pt}=\left(\ln'( T^{\bm{s}}_iT^{\bm{s}}_{i+1}) \pi^{\bs s}_i y_{i-1}'y_{i+1}+(\pi^{\bs s}_i y_{i-1}'y_{i+1})'-
		\pi^{\bs s}_i y_{i-1}'y_{i+1}'\right)/(\pi^{\bs s}_i y_{i-1}'y_{i+1})\Big|_{x=t_j^{i-1}}\\
		&\hspace{30pt}=\sum_{k=1}^{n}\frac{(\lambda^{(k)})^{\bs s}_{[\bs s],i}+(\lambda^{(k)})^{\bs s}_{[\bs s],i+1}+\delta}{t^{(i-1)}_j-z_k}+\sum_{\substack{r=1\\r\neq j}}^{l_{i-1}}\frac{2}{t^{(i-1)}_j-t^{(i-1)}_r}.
		\end{align*}
(Note here that the $t^{(i-1)}_j$ are all distinct, by the assumption that $\bs y^{[i]}$ is generic.) 
	The difference of the right-hand side and the left-hand side is exactly the difference between \eqref{t_i} and \eqref{t_i-1}. 
	
	The BAE for $\wt{\bs t}$ related to  $t^{(i+1)}_j$ is proved by a similar computation. 
		
	All other cases are similar, we omit further details. 
	\end{proof}
	If $s_i=s_{i+1}$, then starting from $\bs y$ we construct a family of new sequences $\bm{y}^{[i]}$, isomorphic to $\C$, by using \eqref{brp}. We call this construction the \emph{bosonic reproduction procedure in $i$-th direction}. 
	If $s_i\neq s_{i+1}$, and $T_i^{\bs s}T_{i+1}^{\bs s}y_{i-1}\neq cy_{i+1}$, $c\in\C^\times$, 
	then starting from $\bs y$ we construct a single new sequence $\bm{y}^{[i]}$  by using \eqref{frp}. We call this construction the \emph{fermionic reproduction procedure in $i$-th direction}. From the definition of fermionic reproduction procedure, $(\bm{y}^{[i]})^{[i]}=\bm{y}$. 
	
	If $\bm{y}^{[i]}$ is  generic with respect to $\bm{s}^{[i]}$, $\bm{\lambda}^{[i]}$, and $\bm{z}$, then by Theorem \ref{GRP}, we can apply the reproduction procedure again.

	Bosonic reproduction procedures fix parity sequences, while fermionic reproductions procedures change parity sequences. Denote by 
\be P_{(\bm{y},\bs{s} )}\subset (\mathbb{P}(\C[x]))^{M+N-1}\times S_{M|N}\nn\ee 
the closure of the set of all pairs $(\tilde{\bs y}, \tilde{\bs s})$ obtained from the initial pair $(\bs y,\bs s)$ by repeatedly applying all possible reproductions. We call $P_{(\bs y,\bs s)}$  the \emph{$\gl_{M|N}$ population of solutions of the BAE associated to $\bm{s}$, $\bm{z}$, $\bm{\lambda}$, and $\bm{l}$, originated at $\bm{y}$}.
By definition, $P_{(\bs y,\bs s)}$ decomposes as a disjoint union over parity sequences, 
\[P_{(\bs y,\bs s)}=\bigsqcup_{\wt{\bm{s}}\in S_{M|N}}P^{\tilde{\bs s}}_{(\bs y,\bs s)},  \qquad
P^{\wt{\bs s}}_{(\bs y,\bs s)}  = P_{(\bm{y},\bs{s} )} \cap \left((\mathbb{P}(\C[x]))^{M+N-1}\times \{\wt{\bs s}\}\right). \]  

\subsection{Rational pseudodifferential operator associated to population}
We define a rational pseudodifferential operator  which does not change under the reproduction procedure.

Let $\bm{s}\in S_{M|N}$ be a parity sequence. Let $\bm{z}=(z_1,\dots,z_n)$ be a sequence of pairwise distinct complex numbers. Let $\bm{\lambda}=(\lambda^{(1)},\dots,\lambda^{(n)})$ be a sequence of polynomial $\gl_{M|N}$ weights. The sequence $\bm{T}^{\bm{s}}=(T^{\bm{s}}_1,\dots,T^{\bm{s}}_{M+N})$ is given by \eqref{T fun}.

Let $\bm{y}=(y_1,\dots,y_{M+N-1})$ be a sequence of polynomials. Recall our convention that $y_0=y_{M+N}=1$. Define a rational pseudodifferential operator $R$ over $\C(x)$, \begin{equation}\label{RPDO of Pop}R^{\bm{s}}(\bs y)=\left(\partial-s_1\ln'\frac{T_1^{\bs s}y_0}{y_1}\right)^{s_1}\left(\partial-s_2\ln'\frac{T_2^{\bs s}y_1}{y_2}\right)^{s_2}\,\dots\,\left(\partial-s_{M+N}\ln'\frac{T_{M+N}^{\bs s}y_{M+N-1}}{y_{M+N}}\right)^{s_{M+N}}.
\end{equation}

The following theorem is the main result of this section. 
\begin{thm}\label{Diff of Popu}
	Let $P$ be a $\gl_{M|N}$ population. Then the rational pseudodifferential operator $R^{\bm{s}}(\bs y)$ does not depend on the choice of ${\bs y}$ in $P$.
\end{thm}
\begin{proof}
	We want to show \begin{align*}\label{localfactor}
	\bigg(\partial-s_i\ln'\frac{T^{\bm{s}}_iy_{i-1}}{y_i}\bigg)^{s_i}&\bigg(\partial-s_{i+1}\ln'\frac{T^{\bm{s}}_{i+1}y_{i}}{y_{i+1}}\bigg)^{s_{i+1}}
	=&\bigg(\partial-s_{i+1}\ln'\frac{T^{\bm{s}^{[i]}}_iy_{i-1}}{\wt{y}_i}\bigg)^{s_{i+1}}\bigg(\partial-s_{i}\ln'\frac{T^{\bm{s}^{[i]}}_{i+1}\wt{y}_{i}}{y_{i+1}}\bigg)^{s_{i}}.
	\end{align*} We have four cases, $(s_i,s_{i+1})=(\pm1,\pm1)$. The cases of $s_i=s_{i+1}$ are proved in \cite{MVpopulations}.
	
		Consider the case of $s_i=-s_{i+1}=1$. We want to show
		\begin{align*}
		\bigg(\partial-\ln'\frac{T^{\bs s}_iy_{i-1}}{y_i}\bigg)\bigg(\partial-\ln'\frac{y_{i+1}}{T^{\bs s}_{i+1}y_{i}}\bigg)^{-1}=\bigg(\partial-\ln'\frac{\wt{y}_{i}}{T^{\bs s}_{i+1}\pi^{\bs s}_i y_{i-1}}\bigg)^{-1}\bigg(\partial-\ln'\frac{T^{\bs s}_{i}(\pi^{\bs s}_i)^{-1}\wt{y}_{i}}{y_{i+1}}\bigg).
		\end{align*} 
		This equation is proved by a direct computation using  \eqref{Ore} and \eqref{frp}. We only note that the rational function $T^{\bm{s}}_iT^{\bm{s}}_{i+1}y_{i-1}y_{i+1}^{-1}$ is not constant by the assumption that the reproduction is possible.

		The case of $s_i=-s_{i+1}=-1$ is similar.
\end{proof}
We denote the rational pseudodifferential operator corresponding to a population $P$ by $R_P$. 

\medskip

It is known that the Gaudin Hamiltonians acting in $L(\bs \la)$  can be included in a natural commutative algebra 
$\mc B(\bs \la)$ of higher Gaudin Hamiltonians, see \cite{MR14}. We expect that similar to the even case, the rational pseudodifferential operator $R^{\bm{s}}(\bs y)$ encodes the eigenvalues of the algebra $\mc B(\bs \la)$ acting on the Bethe vector corresponding to $\bs y$. Then, Theorem \ref{Diff of Popu} would assert that the formulas for the eigenvalues of $\mc B(\bs \la)$ do not depend on the choice of $\bs y$ in the population.

Here we show that the eigenvalues \eqref{hamiltonian} of the (quadratic) Gaudin Hamiltonians do not change under the $\gl_{M|N}$ reproduction procedure. Denote the eigenvalues of the Gaudin Hamiltonians given in \eqref{hamiltonian} by $E_k(\bs y)$, $k=1,\dots,n$. Note that $E_k(\bs y)$ is defined only if $y_i(z_k)\neq 0$ whenever $T_i^{\bs s}(T_{i+1}^{\bs s})^{-s_is_{i+1}}$ vanishes at $x=z_k$. We call such sequences $\bs y$  \emph{$k$-admissible}.

\begin{lem}
Let $\bs y=(y_1,\dots,y_{M+N-1})$ be a sequence of polynomials such that there exists a sequence of polynomials $\wt{\bs y}$
satisfying \eqref{brp} if $s_i=s_{i+1}$ or \eqref{frp} if $s_i=-s_{i+1}$.  Suppose that $\bs y$ and $\wt{\bs y}$ are $k$-admissible. Then $E_k(\bs y)=E_k(\wt{\bs y})$.  
\end{lem}
\begin{proof}
	In the case of $s_i=s_{i+1}$, the lemma follows from $\ln'y_i(z_k)=\ln'\wt{y}_i(z_k)$, $k=1,\dots,n$.
	
	In the case of $s_i\neq s_{i+1}$, the lemma follows from taking logarithmic derivative of the equation \eqref{frp} for $y_i$ and evaluating at $x=z_k$, $k=1,\dots,n$, cf. proof of Lemma \ref{gl11wf}. We only note that by \eqref{frp} the polynomial $y_{i-1}y_{i+1}$ does not vanish at $z_k$ 
	if $T_iT_{i+1}$ does and $y_i$,  $\wt{y}_i$ do not.
\end{proof}

\subsection{Example of a population}\label{Pop}
	In what follows, we study the structure of a population. 
	
	Consider $\gl_{2|1}$. We have three parity sequences, $\bm{s}_0=(1,1,-1)$, $\bm{s}_1=(1,-1,1)$, and $\bm{s}_2=(-1,1,1)$. 
		
	Let $\bm{\lambda}=(\lambda^{(1)},\lambda^{(2)},\lambda^{(3)})$, where $\lambda^{(i)}=(1,1,0)$, for $i=1,2,3$. Let $\bm{z}=(1,\omega,\omega^2)$, where $\omega$ is a primitive cubic root of unity. We have $\bm{T}=\bm{T}^{\bm{s}_0}=(x^3-1,x^3-1,1)$. 
		Let $\bm{y}=(y_1,y_2)=(1,1)$. 
		\begin{enumerate}
			\item First, apply the bosonic reproduction procedure in the first direction to $\bm{y}$. We have $\bm{s}_0^{[1]}=\bm{s}_0$, $\bm{T}^{\bs s_0}=\bm{T}$, and $\bm{y}^{[1]}=(y_1^{[1]},y_2^{[1]})=(x-c,1)$, where $c\in\C P^1$. At $c=\infty$, $\bm{y}^{[1]}=(1,1)=\bm{y}$.
			\item Second, apply the fermionic reproduction procedure in the second direction to $\bm{y}^{[1]}$. We have $(\bm{s}_0)^{[2]}=\bm{s}_1$ and  $\bs T^{\bs s_1}=(x^3-1,x^3-1,1)$.  We have $(\bm{y}^{[1]})^{[2]}=(x-c,4x^3-3cx^2-1)$.
			\item Third, apply the fermionic reproduction procedure in the first direction to $(\bm{y}^{[1]})^{[2]}$. We have  $(\bm{s}_1)^{[1]}=\bm{s}_2$ and $\bs T^{\bs s_2}=\left((x^3-1)^{2},1,1\right)$. We have $((\bm{y}^{[1]})^{[2]})^{[1]}=(2x^4+x,4x^3-3cx^2-1)$.
		\end{enumerate}

		It is easy to check that all further reproduction procedures cannot create a new sequence. Therefore the $\gl_{2|1}$-population $P_{(1,1)}$ is the union of three $\C P^1$, $P_{(1,1)}^{\bm{s}_0}=\{(x-c,1)\;|\;c\in\C P^1\}$, $P_{(1,1)}^{\bm{s}_1}=\{(x-c,4x^3-3cx^2-1)\;|\;c\in\C P^1\}$, and $P_{(1,1)}^{\bm{s}_2}=\{(2x^4+x,4x^3-3cx^2-1)\;|\;c\in\C P^1\}$.
		
		We have the following representations for the rational pseudodifferential operator of the population: $R_P=R^{\bs s_0}=R^{\bs s_1}=R^{\bs s_2}$:
		\begin{align*}
			R_P=\bigg(\partial-\frac{3x^2}{x^3-1}\bigg)\bigg(\partial-\frac{3x^2}{x^3-1}\bigg)\partial^{-1}=\bigg(\partial-\frac{3x^2}{x^3-1}\bigg)\bigg(\partial-\frac{2x^3-3cx^2+1}{x^4-cx^3-x+c}\bigg)\partial^{-1}\\
			=\bigg(\partial-\ln'\frac{x^3-1}{x-c}\bigg)\bigg(\partial-\ln'\frac{4x^3-3cx^2-1}{(x^3-1)(x-c)}\bigg)^{-1}\bigg(\partial-\ln'(4x^3-3cx^2-1)\bigg)\\
			=\bigg(\partial-\ln'\frac{2x^4+x}{(x^3-1)^2}\bigg)^{-1}\bigg(\partial-\ln'\frac{2x^4+x}{4x^3-3cx^2-1}\bigg)\bigg(\partial-\ln'(4x^3-3cx^2-1)\bigg).
		\end{align*}


\section{Populations and flag varieties}\label{sec: spaces}
We call a sequence $\bm{\lambda}=(\lambda^{(1)},\dots,\lambda^{(n)})$ of polynomial $\gl_{M|N}$ weights \emph{typical} if at least one of the $\lambda^{(k)}$, $k=1,\dots,n$, is typical.  
In this section, we show that $\gl_{M|N}$ populations associated to typical $\bs \la$ are isomorphic to the variety of the full superflags.

\subsection{Polynomials $\pi_{a,b}$}
Let $\bs M=(m_1\leq m_2\leq \dots \leq m_r)$, $\bs N=(n_1\leq n_2\leq \dots \leq n_r)$, $m_i,n_i\in\Z$, be two generalized partitions with $r$ parts. We say $\bs N$ \emph{dominates} $\bs M$ if $n_i\geq m_i$ for $i=1,\dots, r$. This gives a partial ordering on the set of 
generalized partitions with $r$ parts.

For a generalized partition with $r$ parts $\bs M$, there exists a unique generalized partition $\bar {\bs M}$ with $r$ parts such that:
\begin{enumerate}
\item all parts of $\bar {\bs M}$ are distinct;
\item $\bar {\bs M}$ dominates $ {\bs M}$; and
\item if a generalized partition with $r$ distinct parts $\bs M'$ dominates $\bs M$, then $\bs M'$ dominates $\bar{\bs M}$.
\end{enumerate}
 We call $\bar {\bs M}$ the \emph{dominant of $\bs M$}. 

We identify multisets of integers with generalized partitions (by putting their elements into weakly increasing order).
\begin{exmp}
Let $\bs M=\{-3,-3,-3,-1,0,5,5,6\}$. Then $\bar {\bs M}=\{-3,-2,-1,0,1,5,6,7\}$.

\end{exmp}
This definition is motivated by the following observation.

Let $V$ be a $d$-dimensional space of functions of $x$ meromorphic around $x=z$ for some $z\in\C$. 
Then $m\in\Z$ is an \emph{exponent of $V$ at $z$} if there is a function $f(x)\in V$ such that the order of the zero at $x=z$ is $m$: 
$f(x)=(x-z)^m(c+o(x-z))$, $c\in\C^\times$. 
Then $V$ has $d$ distinct exponents. We denote $e(V,z)$ the set of exponents of $V$ at $z$.

Let $V_1,\dots,V_k$ be spaces of functions of $x$ meromorphic around $x=z$, $\dim V_i=d_i$. Let $\bs M=\sqcup_{i=1}^k e(V_i,z)$.
Let $V=\oplus_{i=1}^kV_i$. Assume that $\dim V=\sum_{i=1}^k d_i$. Then $e(V,z)$ dominates $\bar {\bs M}$. Moreover, generically, $e(V,z)=\bar {\bs M}$.

\medskip

Let $T_1,\dots ,T_{M+N}\in \C(x)$ be rational functions such that $T_i/T_{i+1}\in\C[x]$ is a polynomial for all $i=1,\dots,M+N-1$, $i\neq M$.
Let $\tau_i(z)$ be the order of the zero of $T_i$ at $x=z$. Set 
$$m_{i}(z)=\tau_{M-i+1}(z)+i-1, \ i=1,\dots,M,\qquad n_i(z)=-\tau_{M+i}(z)+i-1,\ i=1,\dots,N.$$ 
We have $m_1(z)<m_2(z)<\dots <m_M(z)$, $n_1(z)<n_2(z)<\dots <n_N(z)$.

Let $a\in\{0,\dots,M\}$, $b\in\{0,\dots,N\}$. Let $\bar M_{a,b}=\{c_1(z)<\dots<c_{a+b}(z)\}$ be the dominant of $\{m_1(z),\dots,m_a(z),n_1(z),\dots,n_b(z)\}$.
Define
$$d_{a,b}(z)=ab-\sum_{i=1}^{a+b}c_i(z)+\sum_{i=1}^a m_{i}(z)+\sum_{i=1}^b n_{i}(z)=\binom{a+b}{2}-\sum_{i=1}^{a+b}c_i(z)+\sum_{i=1}^a \tau_{M+1-i}(z)-\sum_{i=1}^b \tau_{M+i}(z). $$
Note that $d_{a,b}(z)\geq 0$. Moreover, for all but finitely many $z$ we have $m_i=i-1$, $n_i=i-1$, $c_i=i-1$, and $d_{a,b}(z)=0$. 

We set
\be\label{pi}
\pi_{a,b}=\prod_{z\in\C}(x-z)^{d_{a,b}(z)}.
\ee
Note that $\pi_{a,b}\in \C[x]$ is a polynomial. 

Note that for any non-zero rational function $f(x)$, the polynomials $\pi_{a,b}$ computed from $T_i$ and $fT_i$ are the same.

\subsection{Properties of $\pi_{a,b}$}

Let $\bm{\lambda}$ be a sequence of polynomial $\gl_{M|N}$ weights. Let $T_i=T_i^{\bs s_0}$ be the corresponding polynomials, see 
\eqref{T fun}. Let $\pi_{a,b}$ be the polynomials given by \eqref{pi}.

Let $\bs s$ be a parity sequence. Using $\pi_{a,b}$, the polynomials $T_i^{\bs s}$ can be written in terms of the $T_i$.

\begin{thm}\label{change of T}
	We have
	\[T^{\bs s}_i=
	T_{\sigma_{\bs s}(i)}\ \frac{\pi_{\bs s^+_i,\bs s^-_i}}{\pi_{\bs s^+_i+1,\bs s^-_i}},\; {\rm if }\ s_i=1\quad {\rm and } \quad 
	T^{\bs s}_i=T_{\sigma_{\bs s}(i)}\ \frac{\pi_{\bs s^+_i,\bs s^-_i+1}}{\pi_{\bs s^+_i,\bs s^-_i}},\;  {\rm if }\ s_i=-1.
	\]
\end{thm}
\begin{proof}
Let $\bs s$ be a parity sequence such that $s_i\neq s_{i+1}$ and $\wt{\bs s}=\bs s^{[i]}=(s_1,\dots,s_{i+1},s_i,\dots, s_{M+N})$.
Let $a=\bs s^+_i$, $b=\bs s^-_i+1$. By Lemma \ref{change in T} it is sufficient to check 
	$$
	\frac{\pi_{a+1,b}\, \pi_{a,b-1}}{\pi_{a,b}\, \pi_{a+1,b-1}}=\pi\bigg(T_{M+b}T_{M-a}\frac{\pi_{a,b}}{\pi_{a+1,b-1}}\bigg).
	$$
	
	Since $\lambda^{(k)}$ is a polynomial $\gl_{M|N}$-weight, the exponent of $\pi_{a,b}$ at $z_k$, $d_{a,b}(z_k)$, is given by
	$$
	d_{a,b}(z_k)=\begin{cases}
	ab &{\rm if}\ b\leq \lambda^{(k)}_M,\\
	(a-1)b+\lambda^{(k)}_{M} &{\rm if}\ \lambda^{(k)}_{M}<b\leq \lambda^{(k)}_{M-1},\\
	\dots & \dots\\
	\lambda^{(k)}_{M}+\dots+\lambda^{(k)}_{M-a+1} &{\rm if}\ \lambda^{(k)}_{M-a+1}<b.
	\end{cases}
	$$ 
	Thus the exponent of $\pi_{a+1,b}/\pi_{a,b}$ at $z_k$ is given by 
	$$
	d_{a+1,b}(z_k)-d_{a,b}(z_k)=\min\{b, \lambda^{(k)}_{M-a}\}.
	$$
	The exponent of $(\pi_{a+1,b}\, \pi_{a,b-1})/(\pi_{a,b}\, \pi_{a+1,b-1})$ at $z_k$ is $1$ if $b\leq \lambda^{(k)}_{M-a}$ and it is $0$ otherwise.
	
	To compute the exponent of $T_{M+b}T_{M-a}{\pi_{a,b}}/{\pi_{a+1,b-1}}$ at $z_k$, introduce two extra parameters $c_1,c_2$: $\la^{(k)}_{M-c_1+1}<b-1=\la^{(k)}_{M-c_1}=\dots=\la^{(k)}_{M-c_2+1}<b\leq\la^{(k)}_{M-c_2}$. We have $$d_{a,b}-d_{a+1,b-1}=\begin{cases}
	1+a-b-c_2	&\;{\rm if}\; a\geq c_2,\\
	-\la_{M-a}	&\;{\rm if}\; a<c_2. 
	\end{cases}$$ Note that $\la^{(k)}_{M-a}<b$ implies $\la^{(k)}_{M+b}=0$. A direct computation gives the proof.
\end{proof}

\medskip

Let $W=V\oplus U$ be a graded space of rational functions of dimension $M+N$, where $V=W_{\bar 0}$, $U=W_{\bar 1}$ and $\dim V=M$, $\dim U=N$.
For $z\in \C$, define $m_1(z)< m_2(z)<\dots<m_M(z)$ and $n_1(z)< n_2(z)< \dots < n_N(z)$ to be the exponents of $V$ and $U$ at $z$ respectively. Define rational functions
$$
T_{i}^V=\prod_{z\in\C} (x-z)^{m_{M-i+1}-M+i}, i=1,\dots,M, \quad {\rm and} \quad T_{M+i}^U=\prod_{z\in\C} (x-z)^{-n_i+i-1}, i=1,\dots,N.
$$
Let $\pi^{V,U}_{a,b}$ be polynomials as in \eqref{pi} computed from $T_i^V,T_{M+i}^U$.
The following lemma is clear.

\begin{lem}\label{polynomial 1} Let $v_1,\dots,v_a\in V$, $u_1,\dots, u_b\in U$. Then 
$$\frac{\Wr(v_1,\dots,v_a,u_1,\dots,u_b)\,\pi_{a,b}^{V,U}T_{M+1}^UT_{M+2}^U\dots T_{M+b}^U}{T_M^VT_{M-1}^V\dots T_{M-a+1}^V}
$$ is a polynomial.
\qed
\end{lem}

\subsection{The $\gl_{M|N}$ spaces}\label{glM|N space}
Let $W=V\oplus U$ be a graded space of rational functions of dimension $M+N$, where $V=W_{\bar 0}$, $U=W_{\bar 1}$ and $\dim V=M$, $\dim U=N$.
For $z\in \C$, let as before $m_1(z)< m_2(z)<\dots<m_M(z)$ and $n_1(z)< n_2(z)< \dots < n_N(z)$ be the exponents of $V$ and $U$ at $z$ respectively.

We call $W$ a \emph{$\gl_{M|N}$ space} if the following conditions are satisfied for all $z\in\C$:
\begin{enumerate}
\item $n_N(z)\leq N-1$;
\item if   $m_1(z)<0$, then $m_2(z)\geq 1$, $n_1(z)=m_1(z)$, and $n_i(z)=i-1$, $i=2,\dots, N$;
\item if  $v\in V$, $u\in U$ are not regular at $z$, then there exists a $c\in\C$ such that $(u+cv)(z)=0$.
\end{enumerate}
 
These conditions can be reformulated as follows. 
Let $$p^V=\prod_{z,\, m_1(z)<0} (x-z)^{-m_1(z)}, \qquad p^U=\prod_{z,\, n_1(z)<0} (x-z)^{-n_1(z)}$$ be the least common denominators. 
Then $\bar V = p^VV$ and $\bar U=p^UU$ are spaces of polynomials.  

\begin{lem}\label{equivalent conditions}
The conditions in the definition of the $\gl_{M|N}$ space are equivalent to:
\begin{enumerate}
\item $p^U/p^V$ is a polynomial that is relatively prime with $p^V$;
\item $T_{M-1}^{\bar V}/p^V$ and 
$T_{M+N}^U$
are polynomials;
\item if $T_{M+i}^{\bar U}(z)=0$ for some $i=2,\dots,N$, then $(p^U/p^V)(z)=0$;
\item for any $v\in V, u\in U$, $p^V\Wr(v,u)$ is regular at every zero of $p^V$. 
\end{enumerate}
\end{lem}
\begin{proof}
Let $\tau^V_i(z)$, $\tau^{\bar{V}}_i(z)$, $\tau^U_j(z)$, and $\tau^{\bar{U}}_j(z)$ be the orders of the zeroes of $T^V_i$, $T^{\bar V}_i$, $T^U_j$, and $T^{\bar{U}}_j$ at $z$. If $\tau^V_{M}(z)<0$, then $\tau^{\bar{V}}_i(z)=\tau^V_i(z)-\tau^V_{M}(z)$. If $\tau^U_{M+1}(z)>0$, then $\tau^{\bar{U}}_j(z)=\tau^U_j(z)-\tau^U_{M+1}(z)$.
	
The conditions (1) and (2) in the definition of a $\gl_{M|N}$ space are equivalent to $\tau^U_{M+N}(z)\geq 0$ and if $\tau^V_M(z)<0$, then $\tau^V_{M-1}(z)\geq 0$, $-\tau^U_{M+1}(z)=\tau^V_{M}(z)$, and $\tau^U_{M+2}(z)=\dots=\tau^U_{M+N}(z)=0$. This is equivalent to the first three conditions in the lemma.

The condition (3) in the definition is equivalent to the condition {\it (4)} in the lemma in the presence of the other conditions.
\end{proof}

Let $W=V\oplus U$ be a $\gl_{M|N}$ space. Define polynomials 
\begin{align*}
&T^W_i=T^V_i=\frac{T_i^{\bar V}}{p^V}, \ i=1,\dots, M-1, \qquad T_M^W=p^VT_M^{V}=T_M^{\bar V}, \\
&T_{M+1}^W=\frac{T_{M+1}^U}{p^V}=\frac{p^U}{p^V},\qquad 
T^W_{M+i}={T_{M+i}^U}={p^U}{T_{M+i}^{\bar U}}, \ i=2,\dots, N.
\end{align*}

\begin{rem}
Note that while $T_i^{\bar V}$, $i=1,\dots, M$, are the standard polynomials describing the exponents of the space of polynomials $p^VV$,
our definition of $T_{M+i}^{\bar U}$ has an extra minus sign. The exponents of the space of polynomials $p^UU$ are described by a sequence of polynomials
$(p^U/T_{M+N}^W,\dots, p^U/T_{M+2}^W, 1)$.
\end{rem}

Let $\pi^W_{a,b}$ be as in \eqref{pi} computed from polynomials $T_i^W$.

Further, given $a\in\{0,1,\dots,M\}$, $b\in\{0,1,\dots,N\}$, $v_1,\dots, v_a\in V$, $u_1,\dots, u_b\in U$,
define
\[
y_{a,b}=\frac{\Wr(v_1,\dots,v_{a},u_1,\dots,u_b)\, \pi_{a,b}^W\, p^V\, T^W_{M+1}\dots T^W_{M+b} }{T^W_{M}\dots T^W_{M-a+1}}.
	\]
We have 
\begin{lem}\label{polynomial 2}
	The function $y_{a,b}$ is a polynomial. 
\end{lem}
\begin{proof}
The lemma is proved by considering orders of zeroes at each $z\in\C$. 
\end{proof}
Note that Lemma \ref{polynomial 2} is stronger than Lemma \ref{polynomial 1}, since $y_{a,b}$ has $p^V$ and not $(p^V)^2$ in the numerator. 
Lemma \ref{polynomial 2} holds due to the additional assumption that $W$ is a  $\gl_{M|N}$ space. Here, we crucially use the condition
(3) in the definition of the $\gl_{M|N}$ space.

\medskip

Let $\bs \la=(\la^{(1)},\dots,\la^{(n)})$ be a sequence of polynomial $\gl_{M|N}$ weights, $\bs z=(z_1,\dots,z_n)$ a sequence of pairwise distinct complex numbers. Let $\bs T=(T_1,\dots, T_{M+N})$ be the corresponding polynomials given by \eqref{T fun}. Let $\bs y$ represent a solution of the BAE associated to $\bs \la,\bs z$, and the standard parity sequence $\bs s_0$. We have the rational pseudodifferential operator $R(\bs y)=D_{\bar 0}(\bs y)D_{\bar 1}^{-1}(\bs y)$. Let $V(\bs y)=\ker D_{\bar 0}(\bs y)$, $U(\bs y)=\ker D_{\bar 1}(\bs y)$.
\begin{prop}\label{polynomial glM|N space}
 If $\bs \la$ is typical, then $W(\bs y)=V(\bs y)\oplus U(\bs y)$ is a $\gl_{M|N}$ space of rational functions and $T_i^W=T_i$, $i=1,\dots, M+N$.
\end{prop}
\begin{proof}
Denote $W(\bs y)$, $V(\bs y)$, and $U(\bs y)$ by $W$, $V$, and $U$ respectively.

Note that $y_1,\dots, y_{M-1}$ represents a solution of the $\gl_{M}$ BAE. Therefore, the bosonic reproduction procedures generate a $\gl_{M}$ population and 
$y_M\cdot D_{\bar 0}\cdot(y_M)^{-1}$ is the differential operator associated to this population. It follows by \cite{MVpopulations} that $\bar{V}=y_M V$ is a space of polynomials. Similarly, $y_{M+1},\dots, y_{M+N-1}$ represents a solution of the $\gl_{N}$ BAE and
$\bar{U}=y_M T_{M+1}U$ is also a space of polynomials.
 
We have $p^V=y_M$, $p^U=T_{M+1}y_M$. 

Since $\bs \la$ is typical, there exists $k$ such that $\la^{(k)}$ is typical, i.e. $\lambda^{(k)}_M \geq N$. Then $\la^{(k)}_i+M-i \geq \la^{(k)}_i \geq \lambda^{(k)}_M \geq N >  j-1 \geq - \lambda^{(k)}_{M-j}+j-1$ for $i=1,\dots,M$, $j=1,\dots,N$. Therefore the spaces $V$ and $U$ have no exponents in common and hence $V\cap U=0$. 

The only non-trivial condition in Lemma \ref{equivalent conditions} is {\it (4)}. The fermionic reproduction procedure in the $M$-th direction $\eqref{frp}$ can be written as 
$$
y_M\,\wt y_M=\Wr(v,u)y^2_M\pi_MT_{M+1}/T_M.
$$
Initially, we have $v(\bs y)=T_My_{M-1}/y_M$, $u(\bs y)=y_{M+1}/(T_{M+1}y_M)$. Generic $u,v$ can be obtained from $v(\bs y), u(\bs y)$ by the bosonic reproduction procedures. Therefore, by Theorem \ref{GRP}, $\wt y_M$ is a polynomial for generic $v$, $u$. Since $y_M$ is relatively prime to $\pi_MT_{M+1}/T_M$, we obtain condition {\it (4)} in Lemma \ref{equivalent conditions}.
\end{proof}

\begin{rem} \label{nontyp rem1} If $\bs \lambda$ is not typical then cancellations may occur in the rational pseudodifferential operator $R(\bs y)=D_{\bar 0}(\bs y)D_{\bar 1}^{-1}(\bs y)$ of \eqref{RPDO of Pop} and the spaces $V(\bs y)=\ker D_{\bar 0}(\bs y)$, $U(\bs y)=\ker D_{\bar 1}(\bs y)$ may intersect non-trivially. Compare Lemma \ref{typlem}. As an important example, consider the tensor product of $n$ copies of the defining representation, $L(\bs \lambda) = (\C^{M|N})^n$. Then $T_1(x) = \prod_{k=1}^n (x-z_k)$ and $T_{i}(x) = 1$ for $i=2,\dots,N+M$. Thus for the vacuum solution to the BAE, i.e. $\bs y = (1,\dots,1)$, we have
\be D_{\bar 0}(\bs y) = \left( \partial - \prod_{k=1}^n\frac{1}{x-z_k}\right) \partial^{M-1}, \quad
    D_{\bar 1}(\bs y) = \partial^{N} .\nn\ee
\end{rem}

\subsection{The generating map}
Given a parity sequence $\bs s$ and a full superflag $\mc F\in\mc F^{\bs s}(W)$, we define polynomials $y_i(\mc F)$, $i=1,\dots, M+N-1$, by the formula 
$$
y_i(\mc F)=\begin{cases} y_{\bs s_i^+,\bs s_i^-} \ &{\rm if} \ s_i=1, \\  y_{\bs s_i^+,\bs s_i^-+1} \ &{\rm if} \ s_i=-1.\end{cases}
$$
That defines the \emph{generating map} \[\beta^{\bs s}:\mc F^{\bs s}(W)\to (\mathbb{P}(\C[x]))^{M+N-1}, \quad \mc F\mapsto \bs y(\mc F)=(y_1(\mc F),\dots,y_{M+N-1}(\mc F)).\]

\medskip

Let $\bs \la=(\la_1,\dots,\la_n)$ be a typical sequence of polynomial $\gl_{M|N}$ weights, $\bs z=(z_1,\dots,z_n)$ a sequence of pairwise distinct complex numbers. Let $\bs T=(T_1,\dots, T_{M+N})$ be the corresponding polynomials given by \eqref{T fun}. Let $\bs y$ represent a solution of the BAE associated to $\bs \la,\bs z$ and the standard parity sequence $\bs s_0$. 

Recall that we have the $\gl_{M|N}$ population $P=P_{\bs y}$, see Section \ref{6.1}, the rational pseudodifferential operator of the population $R_P=R(\bs y)=D_{\bar{0}}(\bs y)(D_{\bar{1}}(\bs y))^{-1}$, see \eqref{RPDO of Pop} and  the $\gl_{M|N}$ space $W_P=V(\bs y)\oplus U(\bs y)$, see Proposition \ref{polynomial glM|N space}.

The following theorem asserts that the population $P$ is canonically identified with full superflags $\mc F(W_P)$ and the complete factorizations of the pseudodifferential operator $\mc F(R_P)$.
\begin{thm}\label{operator to population}
For any flag $\mc F\in\mc F^{\bs s}(W_P)$, we have $\beta^{\bs s}(\mc F)\in P^{\bs s}$. Moreover, the generating map 
$\beta^{\bs s}: \mc F^{\bs s}(W_P)\to P^{\bs s}$ is a bijection. Finally, the complete factorization $\rho^{\bs s}(\mc F)$ of $R_P$ 
coincides with $R^{\bs s}(\beta^{\bs s}(\mc F))$, see \eqref{wronskian of rdo1}, \eqref{wronskian of rdo2}, and \eqref{RPDO of Pop}.
\end{thm} 
\begin{proof}
	The operator $R_P^{\bs s_0}$ coincides with the unique minimal fractional decomposition of $R_P$. Thus, for the standard parity, the theorem is proved in \cite{MVpopulations}. 
	
	Let $\bs y=\beta^{\bs s}(\mc F)=(y_1,\dots,y_{M+N-1})$. Lemma \ref{polynomial 2} asserts that $\bs y$ is a sequence of polynomials. By Theorem \ref{change of T}, we have $R^{\bs s }(\bs y)=\rho^{\bs s}(\mc F)$. 
	
	Let $\bs s$ be such that $s_i\neq s_{i+1}$. Let $\wt{\bs s}=\bs s^{[i]}=(s_1,\dots,s_{i+1},s_i,\dots,s_{M+N})$. Let $\wt{\bs y}=\beta^{\wt{\bs s}}(\mc F)=(\wt{y}_1,\dots,\wt{y}_{M+N-1})$. A direct computation shows $y_r=\wt{y}_r$, $r=1,\dots,M+N-1$, $r\neq i$, and $y_i$, $\wt{y}_i$ satisfy equation \eqref{frp}. By Theorem \ref{Diff of Popu} we  have $R^{\wt{\bs s}}(\wt{\bs y})=\rho^{\wt{\bs s}}(\mc F)$.
	
	That reduces the case of any $\bs s$ to the case of $\bs s_0$. 
\end{proof}

\begin{rem} Theorem \ref{operator to population} shows in particular that if two populations intersect, then they coincide.
\end{rem}

Let $W$ be a $\gl_{M|N}$ space. Let $\bs \la_W$ be a sequence of $\gl_{M|N}$ weights and $\bs z_W$ a sequence of distinct complex numbers such that $T_i^W$ are associated to $\bs s_0,\bs \la_W,\bs z_W$.

Let $\bs s$ be a parity sequence. Consider the set of all sequences $(y_1,\dots,y_{M+N-1})\in\beta^{\bs s}(\mc F^{\bs s}(W))$. 
For $i=1,\dots, M+N-1$, let $l^{(\bs s, W)}_i$ be the minimal possible degree of the $i$th polynomial $y_i(x)$ in this set. 

Define 
$$
\la^{(\bs s,\infty)}_W=\sum_{k=1}^n (\la^{(k)}_W)^{\bs s}-\sum_{i=1}^{M+N-1}\al^{\bs s}_i\,l^{(\bs s, W)}_i.
$$

\section{Conjectures and examples}\label{sec: conj}
It is well known that the Bethe ansatz in the naive form is not complete in general.
We conjecture how to overcome this problem. We also give a few examples.

\subsection{Conjecture on Bethe vectors}
Let $\bs \la=(\la^{(1)},\dots,\la^{(n)})$ be a typical sequence of polynomial $\gl_{M|N}$ weights, $\bs z=(z_1,\dots,z_n)$ a sequence of distinct complex numbers. Let $\bs T=(T_1,\dots, T_{M+N})$ be the corresponding polynomials given by \eqref{T fun}.

Let $L(\bs \la)=\otimes_{k=1}^n L(\la^{(k)})$ be the corresponding $\gl_{M|N}$ module.
It is known that the Gaudin Hamiltonians acting in $L(\bs \la)$ can be included in a natural commutative algebra 
$\mc B(\bs \la)$ of higher Gaudin Hamiltonians, see \cite{MR14}. The algebra $\mc B(\bs \la)$ commutes with the diagonal action of $\gl_{M|N}$.

If $N=0$, it is known that the joint eigenvectors of $\mc B(\bs \la)$ in $L(\bs \la)^{sing}$ (up to multiplication by a non-zero constant) are in bijective correspondence with spaces of polynomials $V$, such that $T_i^V=T_i$, see \cite{MTVschubert}. 

Let $\bs s$ be a parity sequence.
We have the following conjecture.

\begin{conj}\label{the conj} The algebra $\mc B(\bs \la)$ has a simple joint spectrum in $L(\bs \la)^{{\bs s}ing}$.
There is a bijiective correspondence between eigenvectors of $\mc B(\bs \la)$ in $L(\bs \la)^{{\bs s}ing}_{\la^{(\bs s,\infty)}}$ (up to multiplication by a non-zero constant) and the $\gl_{M|N}$ spaces of rational functions $W$ such that $T_i^W=T_i$ and $\la^{(\bs s,\infty)}_W=\la^{(\bs s,\infty)}$. Moreover, this bijection is such that, for all $k=1,\dots, n$, the eigenvalue of the Gaudin Hamiltonian $\HH_k$ is given by \eqref{hamiltonian}, where $\bs t$ is represented by any $k$-admissible $\bs y$ in $\beta(\mc F(W))$.
\end{conj}
By simple joint spectrum we mean that if $v_1$,$v_2$ are eigenvectors of $\mc B(\bs \la)$ and $v_1\neq cv_2$, $c\in\C^\times$, then there exists $b\in \mc B(\bs \la)$ such that the eigenvalues of $b$ on $v_1$ and $v_2$ are different.

\begin{rem}\label{non-typical remark}
If the sequence of polynomial modules $\bs \la$ is not typical we expect that the eigenvectors are also parameterized by pairs of spaces of rational functions $V$ and $U$ of dimensions $M$ and $N$ with similar conditions. However, $V$ and $U$ can have a non-trivial intersection (see Remark \ref{nontyp rem1}). Then some fermionic reproduction procedure becomes undefined and the factorization of the rational pseudodifferential operator \eqref{RPDO of Pop} is not minimal. We do not deal with this case here.
\end{rem}

\medskip

In the case of $\gl_{1|1}$, Conjecture \ref{the conj} simplifies as follows. We follow the notation of Section \ref{Reproduction procedure for gl{1|1}}.
Let $\mc N(T)=\ln'(T_1T_2)\pi(T_1T_2)$. 
	
	\begin{conj}\label{conj1} The Gaudin Hamiltonians $\HH_k$, $k=1,\dots,n$, have a simple joint spectrum in $L(\bs \la)^{{\bs s}ing}$.
		There exists a one-to-one correspondence between the monic divisors $y$ of the polynomial $\mc N(T)$ of degree $l$ and the joint eigenvectors $v$ of the Gaudin Hamiltonians of weight $(p-l,q+l)$ (up to multiplication by a non-zero constant).  Moreover, this bijection is such that $\HH_k v=E_kv$, $k=1,\dots,n$, where $E_k$ are given by \eqref{**}.	
	\end{conj}

Recall our conventions from \S\ref{sec: BAE} about what constitutes a solution to the Bethe ansatz equation. With those conventions, a monic divisor of $\mc N(T)$ is the same thing as a solution to the Bethe ansatz equation, cf. Lemma \ref{gl1|1 rplemma}, and in that sense Conjecture \ref{conj1}  asserts that the Bethe ansatz is complete for $\gl_{1|1}$.

		
		\subsection{A $\gl_{1|1}$ example of double roots} Suppose all the tensor factors $L(\lambda^{(k)})$, $k=1,\dots, n$ are non-trivial. In  type $\gl_{1|1}$ that suffices to make them all typical, cf. Remark \ref{gl11rem}. Thus we have $\deg \mc N(T)=n-1$.
	For generic $z$, all roots of the polynomial $\mc N(T)$ are distinct, and there are $2^{n-1}$ different monic divisors of $\mc N(T)$.
	In such a case we have a basis of Bethe eigenvectors in $L(\bs \la)^{{\bs s}ing}$, in accordance with Conjecture \ref{conj1}.	
	But when the polynomial $\mc N(T)$ has multiple roots the number of its divisors is smaller. Then,  according to 
	Conjecture \ref{conj1}, we should expect non-trivial Jordan blocks in the action of the Gaudin Hamiltonians. We give an example illustrating this point.
	
		We consider the case when $n=3$. We work with the standard parity sequence. 
		
		The modules $L(\lambda^{(k)})$, $k=1,2,3$ are spanned by $v^{(k)}_+$ and $v^{(k)}_-$, where $v^{(k)}_+$ is the highest weight vector with respect to $\bs s_0$, and $v^{(k)}_-=e_{21}v^{(k)}_+$. Denote the vector $v^{(1)}_i\otimes v^{(2)}_j\otimes v^{(3)}_k$, $i,j,k\in\{\pm\}$ by $v_{(ijk)}$. Let $h_k=p_k+q_k$, $k=1,2,3$. We are supposing that  $h_k \neq 0$, $k=1,2,3$. 
		We have
		\be\label{N(T)}
     \mc N(T)=(h_1+h_2+h_3)x^2-(h_1(z_2+z_3)+h_2(z_1+z_3)+h_3(z_1+z_2))x+(h_1z_2z_3+h_2z_1z_3+h_3z_1z_2).
		\ee
		The weights $\la^{(i)}$ being polynomial means that $h_i\in\Z_{\geq 1}$. 
		
		The subspace $L(\bm{\lambda})^{\textup{sing}}_{(p-1,q+1)}$ is spanned by $w_1=-h_2v_{(-++)}+h_1v_{(+-+)}$ and $w_2=-h_3v_{(+-+)}+h_2v_{(++-)}$. The action of the Gaudin Hamiltonians in this subspace is explicitly given by
		\begin{align*}
		\HH_1&=\Big(\frac{p_1p_2-q_1q_2}{z_1-z_2}+\frac{p_1p_3-q_1q_3}{z_1-z_3}\Big)I
		+\bmx  -\frac{h_1+h_2}{z_1-z_2} & -\frac{h_3}{z_1-z_2}\\  -\frac{h_2}{z_1-z_3} & -\frac{h_1+h_3}{z_1-z_3} \emx,\\
		\HH_2&=\Big(\frac{p_2p_1-q_2q_1}{z_2-z_1}+\frac{p_2p_3-q_2q_3}{z_2-z_3}\Big)I+\bmx -\frac{h_1+h_2}{z_2-z_1}& \frac{h_1}{z_2-z_3}& \\ \frac{h_3}{z_2-z_1}& -\frac{h_2+h_3}{z_2-z_3}& \emx.
		\end{align*} 
		The discriminants of the characteristic polynomials of both of the above $2\times 2$ matrices coincide with the right-hand side of \eqref{N(T)} up to multiplication by nonzero factors.  Therefore the polynomial $\mc N(T)$ has distinct roots if and only if $\HH_1,\HH_2$ have distinct eigenvalues, that is, if and only if the Gaudin Hamiltonians are diagonalizable.

We note that in the case of double roots of $y(x)$, the corresponding Bethe vector is zero. Therefore an actual eigenvector should be obtained via an appropriate derivative. It can be done in the case of $\gl_{1|1}$ without difficulties, but in general the algebraic procedure is not known.
\subsection{A $\gl_{1|1}$ example with non polynomial modules} 
Conjecture \ref{conj1} may be true for arbitrary modules, not only polynomial ones if we make the following modification.
Let $\bs \la$ be a sequence of arbitrary $\gl_{1|1}$ weights. 
In general $L(\bs \la)$ need not be completely reducible. That is, there may exist a nonzero singular vector $v\in L(\bs \la)^{\bs sing}$ such that $v=e_{21}^{\bs s}\,w$ for some $w\in L(\bs\la)$.     
If $v$ and $w$ are eigenvectors then the eigenvalues of $v$ and $w$ are the same and we do not expect to obtain a new divisor of $\mc N(T)$ for $v$.

\begin{conj}\label{conj} 
Consider the subspace of $L(\bs \la)^{\bs sing}$ spanned by the joint eigenvectors of the Gaudin Hamiltonians $\HH_k$, $k=1,\dots,n$. Quotient it by its intersection with the image of $e_{21}^{\bs s}$. On this subquotient, the Gaudin Hamiltonians $\HH_k$, $k=1,\dots,n$  have a simple joint spectrum and their joint eigenvectors of weight $(p-l,q+l)$ (up to multiplication by a non-zero constant) are in one-to-one correspondence with the  monic divisors $y$ of the polynomial $\mc N(T)$ of degree $l$.  
Moreover, this bijection is such that $\HH_k v=E_kv$, $k=1,\dots,n$, where $E_k$ are given by \eqref{**}.
\end{conj}
Here we give an example of a such a phenomenon.
			We consider the case when $n=3$.	
		Suppose $h_1+h_2+h_3=0$, that is $p+q=0$. Then the polynomial $\mc N(T)$ given by \eqref{N(T)} is linear. In particular, we have only two divisors instead of the four which we had in a generic situation. We denote the only root of $\mc N(T)$ by $t$. 
		
		The subspace $L(\bm{\lambda})_{(p-1,-p+1)}$ is three dimensional. It has a basis $\{w, e_{21}v_{(+++)}, v\}$ where $w$ is any vector such that $e_{12}w=v_{(+++)}$, and the two other vectors		$e_{21}v_{(+++)}=v_{(-++)}+v_{(+-+)}+v_{(++-)}$ and $v=(t-z_1)^{-1}v_{(-++)}+(t-z_2)^{-1}v_{(+-+)}+(t-z_3)^{-1}v_{(++-)}$ 
	are singular.
		
		The subspace $L(\bm{\lambda})_{(p-2,-p+2)}$ is also three dimensional. It has a basis $\{u, e_{21}w,e_{21}v\}$, where $u$ is any vector such that $e_{21} u = v_{(---)}$. One can check that $e_{21}v$ is proportional to $e_{12}v_{(---)}$, and is therefore singular since $e_{12}^2=0$.  		

The structure of the module can be pictured as follows:
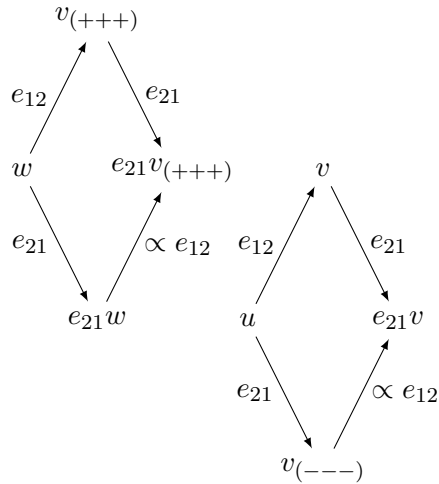
\be\nn\begin{tikzpicture}[yscale=-2,xscale = 1]
\node (vppp) at (0,0) {$v_{(+++)}$};
\node (w) at (-1,1) {$w$};
\node (e21vppp) at (1,1) {$e_{21} v_{(+++)}$};
\node (e21w) at (0,2) {$e_{21} w$};
\node (v) at (3,1) {$v$};
\node (e21v) at (4,2) {$e_{21}v$};
\node (u) at (2,2) {$u$};
\node (vmmm) at (3,3) {$v_{(---)}$};
\draw[-latex] (w) -- node[midway,left] {$e_{12}$}(vppp);
\draw[-latex] (u) -- node[midway,left] {$e_{12}$}(v);
\draw[-latex] (e21w) -- node[midway,right] {$\propto e_{12}$}(e21vppp);
\draw[-latex] (vmmm) -- node[midway,right] {$\propto e_{12}$}(e21v);
\draw[-latex] (vppp) -- node[midway,right] {$e_{21}$} (e21vppp);
\draw[-latex] (v) -- node[midway,right] {$e_{21}$} (e21v);
\draw[-latex] (w) -- node[midway,left] {$e_{21}$} (e21w);
\draw[-latex] (u) -- node[midway,left] {$e_{21}$} (vmmm);
\end{tikzpicture}\ee

While the singular space $L(\bm{\lambda})^{\rm sing}$ is four dimensional, its quotient by the image of $e_{21}$ is  two dimensional and generated by the images of $v_{(+++)}$ and $v$, in accordance with the Conjecture \ref{conj}.
		
Let $\bs s_1=(-1,1)$ be the only non-standard parity sequence. The subspace of $\bs s_1$-singular vectors has a basis $\{v_{(---)}, e_{21}v, e_{21}w, e_{21}v_{(+++)} \}$. Its quotient by the image of $e^{\bs s_1}_{21}$ is generated by images of $v_{(---)}$ and $e_{21}w$. 
	
The reproduction procedure connects $v_{(+++)}$ with $e_{21}w$ and $v$ with $v_{(---)}$. In particular, it connects vectors with the same eigenvalues, see Lemma \ref{gl11wf}; however the weight now changes by $2\alpha$ and Corollary \ref{e_{21}v} is not true in this situation.

\end{document}